\documentclass[a4paper,12pt]{article}

\usepackage[english]{babel}
\usepackage[utf8x]{inputenc}
\usepackage[T1]{fontenc}

\usepackage[a4paper,top=3cm,bottom=2cm,left=3cm,right=3cm,marginparwidth=1.75cm]{geometry}

\usepackage{amsmath}
\usepackage{amssymb}
\usepackage{graphicx}
\usepackage[colorinlistoftodos]{todonotes}
\usepackage[colorlinks=true, allcolors=blue]{hyperref}

\usepackage{marginnote}

\title{An invitation to multisymplectic geometry}
\author{Leonid Ryvkin\thanks{Fakult\"at f\"ur Mathematik, Ruhr-Universit\"at Bochum, Universit\"atsstr. 150, 44801 Bochum, Germany; Leonid.Ryvkin@rub.de}~ and Tilmann Wurzbacher\thanks{Universit\'{e} de Lorraine, C.N.R.S., IECL, F-57000 Metz, France; tilmann.wurzbacher@univ-lorraine.fr}}
\usepackage{amsthm}
\usepackage{amsfonts}
\usepackage{xypic}
\usepackage{enumerate}
\newtheorem{definition}{Definition}[section]
\newtheorem{remark}[definition]{Remark}
\newtheorem{lemma}[definition]{Lemma}
\newtheorem{theorem}[definition]{Theorem}
\newtheorem{corollary}[definition]{Corollary}
\newtheorem{example}[definition]{Example}
\newtheorem{construction}[definition]{Construction}
\newtheorem{proposition}[definition]{Proposition}
\newtheorem{propdef}[definition]{Proposition / Definition}
\newtheorem*{theorem*}{Theorem}
\newtheorem*{definition*}{Definition}
\newtheorem*{ack*}{Acknowledgements}

\begin{document}
\maketitle

\begin{abstract}
In this article we study multisymplectic geometry, i.e., the geometry of manifolds with a non-degenerate, closed differential form. First we 
describe the transition from Lagrangian to Hamiltonian classical field theories, and then 
we reformulate the latter in ``multisymplectic terms''. 
Furthermore, we investigate basic questions on normal forms of multisymplectic manifolds, notably the questions wether and when 
Darboux-type theorems hold, and ``how many''  diffeomorphisms certain, important classes of multisymplectic manifolds
possess. Finally, we survey recent advances in the area of symmetries and conserved quantities on multisymplectic manifolds. 
\end{abstract}

\noindent {\bf MSC-classification (2010):} 53D05, 70S05, 37C05, 53D20, 37K05.\\

\noindent {\bf Keywords:} multisymplectic geometry, Lagrangian and Hamiltonian field theories, Darboux-type theorems, groups of multisymplectic diffeomorphisms, homotopy comoment

\tableofcontents

\section{Introduction}
The quest for a ``Hamiltonian'' formulation of classical field theory has a long history, going back at least to Volterra's work at the end of the 19th century (see \cite{MR2882772} and \cite{MR2892558} for some stimulating historical remarks). The main advantage of such a formalism, as opposed to the ``Lagrangian'' approach is that it replaces the problem of finding critical points of a real-valued functional on a space of maps by the finite-dimensional problem of finding geometrically defined ``integral curves'' or ``vortex curves'', though in a field-theoretic context the latter objects typically have $n$-dimensional sources. Analogous to the case of mechanics, the study of general ``(multi)phase spaces'' (i.e. symplectic resp. multisymplectic manifolds) is not only crucial to understand the dynamics but is also of great independent interest, e.g. for the quantization of mechanics or field theories.\\

This article has a two-fold goal: First, we want to explain in a few pages the transition from Lagrangian classical field theory to dynamics defined by a real-valued function on a multisymplectic manifold. We are lead by the principle, that the Hamiltonian approach to multisymplectic dynamics should be formulated as simple and universal as it is in the case of Hamiltonian dynamics of symplectic manifolds. The second goal of the paper is to ``invite'' the reader to the study of multisymplectic manifolds via a ``guided tour'' that is, in parts, a survey, but which exhibits also many unpublished results and new examples, elucidating properties discussed in the text.\\

Let us describe now the content of this paper in more detail. As alluded to above the first section reviews the Hamiltonian approach to time-dependent mechanics and the transition from the Lagrangian variational approach to a ``Hamiltonian functional'' $\mathbb H$, whose critical points are the solutions of the given classical field theory. Notably, we explain the relevant ``multiphase spaces'' together with their canonical differential forms rigorously and from scratch. We then give several conditions equivalent to being critical for $\mathbb H$, culminating in the following

\begin{theorem*}[Compare Theorem \ref{fin1thm}]
	Let $(\Sigma,\mathrm{vol}^\Sigma)$ be a manifold of dimension $n$ with a fixed volume form, $\gamma\in\mathfrak {X^n}(\Sigma)$ the multivector field defined by $\iota_{\gamma}(\mathrm{vol}^\Sigma)=1$ and $\Psi:\Sigma\to M(\pi)$ a section of the multiphase space, viewed as a bundle over $\Sigma$. Then $\Psi$ is critival for $\mathbb H$ if and only if $$\forall x\in \Sigma, (\Psi_*)_x(\gamma_x)=X_H(\Psi(x)),$$ where the multivector field $X_{H}$ on $M(\pi)$ fulfills the ``Hamilton-deDonder-Weyl equation'' $\iota_{X_H}\omega=(-1)^ndH$, with $\omega$ being the canonical $(n{+}1)$-form on $M(\pi)$.  
\end{theorem*} 

In Section 2 we start with the definition of a multisymplectic manifold and give a long list of examples, before recalling a fundamental result of Martinet ``explaining'' why there are many multisymplectic manifolds. We conclude this section by generalizing Remark \ref{rem:lag} to an arbitrary multisymplectic manifold (compare Definition \ref{defhddw} and Remark \ref{rem-end-2}):
\begin{definition*}
For a given $n$-plectic manifold $(M,\omega)$ and a $k$-form $H$ on $M$, an $(n{-}k)$-vector field $X_H$ solves the ``Hamilton-DeDonder-Weyl equation'' if $$\iota_X\omega=-dH.$$ Furthermore, if $\Sigma$ is an $n{-}k$-dimensional manifold, $\gamma\in \mathfrak X^{n{-}k}(\Sigma)$ and $\Psi:\Sigma\to M$ satisfies $\forall x\in\Sigma$, $(\Psi_*)_x(\gamma_x)=X_H(\Psi(x))$ we call $(\Sigma,\gamma,\Psi)$ a ``Hamiltonian $(n{-}k)$-curve for $H$''.
\end{definition*}

Section 3 is devoted to the question of normal forms in multisymplectic geometry. We first report on the linear case, recently completed by the first author (cf. \cite{1609.02184}). The main point is here, of course, that typically there exist several different equivalence classes of non-degenerate $(n{+}1)$-forms, called ``linear types'' on a given finite-dimensional real vector space. We then introduce the basic 

\begin{definition*}
	A multisymplectic manifold $(M,\omega)$ is called flat, if for all $p$ in $M$, there exists a local diffeomorphism beween $M$ and  $T_pM$, identifying $\omega$ with the constant-coefficient form $\omega_p$ on $T_pM$.
\end{definition*}

We describe in Subsections \ref{sub:symp}-\ref{sub:lie} imporant classes of multisymplectic manifolds and their respective flatness conditions. Notably the results in Sections \ref{sub:new1} and \ref{sub:new2} (Theorems \ref{prodthm} and \ref{clxthm}) are new. In Section \ref{subs:2pl6} we give an elementary construction to obtain 2-plectic structures on $\mathbb R^6$, elucidating the two-fold obstruction to flatness: non-constancy of the linear type and an ``integrability condition'', whose details depend on the linear type (then assumed to be constant). We conclude with a short subsection, \ref{sub:lie}, on the canonical 3-form on a Lie group $G$. In Section \ref{sec:2} (compare Example \ref{exlie}) we explained why $(G,\omega)$ is 2-plectic if $G$ is a semi-simple Lie group. Here we show 

\begin{theorem*}
[cf. Theorem \ref{liegrp}]
	Let $G$ be a simple real Lie group and $\omega$ its canonical three-form. Then $(G,\omega)$ has constant linear type, and $(G,\omega)$ is flat if and only if $G$ is three-dimensional (and $\omega$ then a volume form).
\end{theorem*} 

In Section 5 we study the automorphism group of multisymplectic manifolds. In contrast to the groups of diffeomorphisms preserving a symplectic or volume form, the group of multisymplectomorphisms of $(M,\omega)$ tends to be rather ``small'', even if $(M,\omega)$ is flat. We study here notably the question if this group is $k$-transitive $k=1$, $k=2$ or for all $k\in\mathbb N$. Very little seems to be known on these automorphism groups even for simple classes of multisymplectic manifolds. Most of the results and examples of this section seem to be new, but the Theorems \ref{moser} and \ref{lepage}, as well as the ideas of the proof of Theorem \ref{transC} and Lemma \ref{lemfol} are known. The main feature of the study of automorphism groups of multisymplectic manifolds is the preservation of natural distributions or tensors associated to the multisymplectic form, compare Examples \ref{ediff1} and \ref{ediff2}, Remark \ref{rdiff1} and Proposition \ref{pdiff1}. The presence of these unexpected ``invariants'' reduce the size of the symplectomorphism group in a non-trivial manner, leading to uncharted territory. This is even more the case for non-flat multisymplectic manifolds, as is illustrated by the following \\

\begin{theorem*}[cf. Propositions \ref{exclt} and \ref{nflat2}]
	Let $N=\{(x^1,x^2,x^3,x^4,x^5,x^6)\in \mathbb R^6~|~x^2>0\}$ and $\omega^f=  dx^{135} -dx^{146}-dx^{236} + f(x^2)\cdot dx^{245}$, where $f:\mathbb R^{>0}\to \mathbb R^{>0}$ is smooth, and $dx^{ijk}=dx^i\wedge dx^j\wedge dx^k$. Then $(N,\omega^f)$ is multisymplectic and of constant linear type. Furthermore,
	\begin{enumerate}[(i)]
		\item Let $f(x^2)=x^2$, then $(N,\omega^f)$ is non-flat and its multisymplectic diffeomorphism group acts transitively but not 2-transitively on $N$.
		\item Let $f$ satisfy $f|_{]0,1]}=1$ and $f|_{[2,\infty[}(t)=t$, then there
		are open subsets of $N$ where $\omega^f$ is flat resp. non-flat and therefore the group of multisymplectic diffeomorphisms of $(N,\omega)$ can not act transitively on $N$.  
	\end{enumerate}
\end{theorem*}

In Section 6 we discuss the notion of observables and Hamiltonian symmetries on multisymplectic manifolds. It is well-known that both notions lead in field theories to ``algebraic complications'' compared to the case of mechanics. Typically in physical field theories symmetries preserve the Lagrangian density only up to a total divergence and symmetries form Lie algebras ``up to divergences''. Mimicking the Poisson bracket on a multisymplectic manifold, we find a precise ``Hamiltonian'' counterpart of these phenomena: Considering $(n{-}1)$-forms $\alpha,\beta,\gamma$ on an $n$-plectic manifold $(M,\omega)$ possessing vector fields $X_\alpha,X_\beta, X_\gamma$ such that $\iota_{X_\alpha=-d\alpha}$ etc., we define 
\[\{\alpha,\beta\}:= l_2(\alpha,\beta)=\iota_{X_\beta}\iota_{X_\alpha}\omega.\] 
We then find that $l_2$ is a Lie bracket up to exact terms: 

\[
\{\alpha, \{\beta,\gamma\}\}-\{\{\alpha, \beta\},\gamma\}-\{\beta, \{\alpha,\gamma\}\}=-d (\iota_{X_\gamma}\iota_{X_\beta}\iota_{X_\alpha}\omega).
\]
Mathematicians and physicists tended to mod out closed or exact forms in order to get a bona fide Jacobi identity in similar contexts but under the influence of Stasheff and others, more (``higher'') structure became acceptable and these terms were kept in the picture. Baez and Rogers (cf. \cite{MR2892558}) finally uncovered the fact that a natural choice of observables on a multisymplectic manifolds carries the structure of a Lie $\infty$-algebra. In Subsection \ref{obs1} we recall its definition and give several examples of observable algebras. The natural next step was to define a comoment map as an $L_\infty$-morphism from a Lie algebra to the observables (compare \cite{MR3552542}). We review in Subsection \ref{obs2} this concept and the characteristic classes associated to a multisymplectic Lie algebra action obstructing the existence of a comoment. In Subsection \ref{obs3} we report, without giving proofs, on recent results on conserved quantities with respect to a Hamiltonian vector field, fulfilling $\iota_{X_H}\omega=-dH$ for an $(n{-}1)$-form $H$ on an $n$-plectic manifold $(M,\omega)$.

\begin{ack*}
	We would like to thank Daniel Bennequin, Fr\'ed\'eric H\'elein, Frank Kutzschebauch, Camille Laurent-Gengoux and Marco Zambon for discussions related to this article. The first-named author would also like to thank the RUB Research School for financial support.
\end{ack*}

\section{Classical Field Theory}
\label{sec:1}
In this section we explain the Hamiltonian formulation of classical field theories that allows to replace the infinite dimensional (Lagrangian) variational approach by the study of analytic and geometric questions on certain types of finite dimensional manifolds, called multiphase spaces or multimomentum bundles. Since these ideas are a direct, though technically involved generalization of time-dependent Hamiltonian mechanics, we start by reviewing the latter subject in Subsection \ref{subsec1}. In the next subsection, \ref{subsec1-2}, we describe the general set-up, i.e. jet bundles and multiphase spaces in some detail. In Subsection \ref{subsec1-3} we explain the transition from Lagrangian to Hamiltonian classical field theories, and the various equivalent characterizations of solutions of field theories. The content of this section is essentially known if not classic and good references include \cite{MR997295} for the mechanics part and \cite{ MR1244450, MR2370237, MR0341531,MR2105190, MR0334772, MR1098517,1310.7930} and \cite{MR2559661} for the field theory part, but we think that its inclusion is highly useful here. It allows to see how multisymplectic manifolds and the Hamilton-Volterra (and Hamilton-DeDonder-Weyl) equations generalize the well-known ``Hamiltonian picture'' of mechanics to field theory. The main novelties here are our insisting on the question how the dynamics of a classical field theory can be defined by a ``Hamiltonian function'', as opposed to a ``Hamiltonian section'' and the introduction of ``vortex n-planes'' inside multiphase spaces to characterize the solutions of a Hamiltonian field theory. Put together this allows to formulate the condition on a map with $n$-dimensional source to be a solution in a geometric way that generalizes from multiphase spaces to arbitrary ``n-plectic manifolds'' (see the second condition in Theorem \ref{fin1thm}, the remarks following the theorem and Remark \ref{rem-end-2}).

\subsection{Time-dependent classical mechanics revisited}
\label{subsec1}
Since Lagrangian formulations typically allow for an explicit dependence on time or spacetime of the density, we review here thoroughly the Hamiltonian approach to time-dependent mechanics.

\begin{definition}
	\begin{enumerate}	Let $Q$ be a manifold of dimension $N$ and $T^*Q$ its cotangent bundle.
		\item  The 1-form $\theta^{T^*Q}$ on $T^*Q$ defined by $$\theta_{\alpha_{q}}(u_{\alpha_{q}}):=\alpha_q(  (\mathrm{proj}_Q)_* (u_{\alpha_q})),~~~~~~\forall q\in Q, \forall \alpha_q\in T^*_qQ,~~ \forall u_{\alpha_q}\in T_{\alpha_q}(T^*Q),$$ 
		is called the  tautological (or canonical) 1-form on $T^*Q$. The negative of its exterior derivative $\omega^{T^*Q}=-d\theta^{T^*Q}$ is called the canonical 2-form on $T^*Q$. In this context, $Q$ is sometimes called the ``configuration space'' and $T^*Q$ the ``phase space (associated to $Q$)''. 
		\item Given local coordinates $(q^1,...,q^N)$ on $Q$, defined on an open subset $U\subseteq Q$, we can describe an element $\alpha\in T^*U\subseteq T^*Q$ by its basepoint $q\in U\subseteq Q$ and its coefficients relative to the base $\{  (dq^1)_q, ...,(dq^N)_q   \}$ of $T^*_qQ$. I.e., given the coordinates $q^a$ on $U$ the standard coordinates of $\alpha=\sum_{a=1}^{N}p_a(dq^a)_q$ are $(q^1,...,q^N,p_1,...,p_N)$.
		\item Given a symplectic manifold $(M,\omega)$ and a (``Hamiltonian'') function $H\in C^\infty(M,\mathbb R)$, we call the unique vector field $X_H\in \mathfrak X(M)$ fulfilling $\iota_{X_H}\omega=-dH$ the ``Hamiltonian vector field associated to $H$''.  
		\item The equation $X_H(\gamma(t))=\dot \gamma(t)$ for a differentiable curve $\gamma$ in a symplectic manifold $(M,\omega)$ is called the ``Hamiltonian equation (for $\gamma$ with respect to the Hamiltonian function $H$)''. In local coordinates $(q^a,p_a)$ satisfying $\omega=\sum_{a=1}^Ndq^a\wedge dp_a$ one arrives at the traditional Hamiltonian equation
		$$
		\frac{d(q^a\circ \gamma(t))}{dt}=-\frac{\partial H}{\partial p_a}(\gamma(t)), ~~\frac{d(p^a\circ \gamma(t))}{dt}=\frac{\partial H}{\partial q^a}(\gamma(t)) ~~~\forall a\in\{1,...,N\}.
		$$
		\item Let $I\subset \mathbb R$ be an open interval and $\mathcal H:I\times T^*Q\to \mathbb R$ be a smooth function, called a ``time-dependent Hamiltonian function''. Identifying $T^*(I\times Q)$ with $\mathbb R\times( I\times T^*Q)$ and denoting by $p$ the standard coordinate on $\mathbb{R}$, the fiber of $T^*I\to I$, the submanifold $$W:=\{\mathcal H-p=0\}\subset T^*(I\times Q)$$ is a smooth hypersurface, the image of $h:I\times T^*Q\to T^*(I\times Q)=\mathbb R\times (I\times T^*Q),$ where $ h=(\mathcal H, id_{I\times T^*Q})$. The form $\omega^{T^*(I\times Q)}$ on $W$ has one-dimensional kernel $\mathrm{ker}(\omega^{T^*(I\times Q)}|_{TW})\subset TW$, and a leaf of the associated foliation is called a ``vortex line of $\mathcal H$ in $T^*(I\times Q)$''.
		
		 (Note that, e.g. in \cite{MR997295}, one can equivalently consider the isomorphism $h:I\times T^*Q\to W$ and $\omega_h=h^*(\omega^{T^*(I\times Q)})$ and interpret a vortex line as a one-dimensional submanifold of $I\times T^*Q$.)
	\end{enumerate}
\end{definition}

\noindent We can now formulate the folkloric

\begin{theorem}[Equivalent formulations of time-dependent dynamics in Hamiltonian mechanics]
	Let $Q$ be an $N$-dimensional manifold, $I\subset \mathbb R$ an open interval and $\mathcal H:I\times T^ *Q\to \mathbb R$ a smooth function. Then the following are equivalent for a smooth map $\psi:I\to T^ *Q$:
	\begin{enumerate}
		\item The map $\psi$ satisfies the (time-dependent) Hamiltonian equation for $\mathcal H$ on $(T^ *Q, \omega^ {T^*Q})$, i.e. one has $X_{\mathcal H_t}(\psi (t))=\dot{\psi}(t)$, where $\mathcal H_t(y):=\mathcal H(t,y)$ $\forall t\in I, \forall y \in T^ * Q$. Equivalently, in standard coordinates on $T^ *Q$ one has for all $a\in \{1,...,N\}$:
		$$
		\frac{\partial\mathcal H}{\partial q^a}(t,\psi(t))
		=\frac{d(p_a\circ \psi(t))}{dt}, 
		~~-\frac{\partial\mathcal H}{\partial p_a}(t,\psi(t))
		=\frac{d(q^a\circ \psi(t))}{dt}.
		$$
		\item The map $\Psi: I\to T^ *(I\times Q)=\mathbb R \times (I\times T ^ *Q)$, defined by $\Psi(t)=( \mathcal H(t,\psi(t)),t, \psi(t))$, satisfies the Hamiltonian equation for $H=\mathcal H-p$ on $(T^*(I\times Q), \omega^ {T^* (I\times Q)})$, i.e. $ \dot \Psi(t) = X_H(\Psi(t))$ $\forall t$. Writing $x$ for the canonical coordinate on $I$ and $p$ for the canonical coordinate on $T^*_xI$, one has in standard coordinates $\forall a\in \{ 1,...,N\}$:
		$$
		\frac{\partial H}{\partial q^a}(\Psi(t))=	\frac{\partial\mathcal H}{\partial q^a}(t,\psi(t))
		=\frac{d(p_a\circ \psi(t))}{dt},$$
		$$
		-\frac{ H}{\partial p_a}(\Psi(t))=-\frac{\partial\mathcal H}{\partial p_a}(t,\psi(t))
		=\frac{d(q^a\circ \psi(t))}{dt},
		$$		
		as well as
		$$
		 \frac{\partial H}{\partial x}(\Psi(t))=\frac{d(p\circ \Psi(t))}{dt} ~\mathrm{and} ~ -\frac{\partial H}{\partial p}(\Psi(t))=\frac{d(x\circ \Psi(t))}{dt}.
		$$
		\item For the above $\Psi$, $im(\Psi)$ is a vortex line of $\mathcal H$ in $T^ * (I\times Q)$.
		\item For all $X\in \mathfrak X(T^*Q)$ with compact support, considered as ``vertical vector fields'' on $I\times T^ *Q\to I$ (i.e. vector fields in $\mathrm{ker}(proj_I)_*\subset T(I\times T^ *Q)$), we have $(id_I, \psi)^*(\iota_X\omega_h)=0$, where $\omega_h=h^ *(\omega^{T^*(I\times Q)})$. In standard coordinates $\omega_h$ is given by 
		$$
		-d\mathcal H\wedge dx - \sum_{a=1}^Ndp_a \wedge dq^a. 
		$$ 
		\item The section $(id_I, \psi)\in \Gamma_{C^\infty}(I,I\times T^*Q)\cong  {C^\infty}(I,T^*Q)$ is a critical point of the functional 
		\[
		\mathbb H: \Gamma_{C^\infty}(I,I\times T^*Q)\to \mathbb R, ~~\mathbb H[\psi]:=\int_I(id_I,\psi)^*\theta_h,
		\]
		where $\theta_h=h^*(\theta^{T^*(I\times Q)})$.
	\end{enumerate}
\end{theorem}

\begin{proof}
	Obviously the second assertion implies the first. Assume now that statement 1 is true. Since $x(t)=t$, $\frac{dx}{dt}=1$ and by $-\frac{\partial H}{\partial p}=\frac{\partial p}{\partial p}=1$ the second ``new'' equation in statement 2 is verified. Furthermore, since $H=\mathcal H - p$ is constant on the solution curves and $\psi$ solves the Hamiltonian equation for $\mathcal H$:
	{\small \[
	\frac{dp}{dt}=\frac{d\mathcal H}{dt}=\sum_a\left( \frac{\partial \mathcal H}{\partial q^a}\frac{dq^a}{dt} + \frac{\partial H}{\partial p^a}\frac{dp^a}{dt} \right) +\frac{\partial \mathcal H}{\partial x}\frac{dx}{dt}
	=\sum_a\left( \frac{dp_a}{dt}\frac{dq^a}{dt} - \frac{dq^a}{dt}\frac{dp^a}{dt} \right) +\frac{\partial \mathcal H}{\partial x}=\frac{\partial \mathcal H}{\partial x}.
	\]}
	This implies, that the first new equation is satisfied as well.\\
	
	\noindent Assume now, that the second assertion is satisfied and we are given \[t_0\in I, \Psi(t_0)\in \mathrm{im}(\Psi)\subset W=\{ H=\mathcal H - p=0 \}\subset T^ *(I\times Q).\] Then $\dot \Psi (t_0)$ generates $T_{\Psi(t_0)}\mathrm{im}(\Psi)$ and we have \[\forall u\in T_{\Psi(t_0)}(W)=T_{\Psi(t_0)}(\{H=0\})=\{  v\in T_{\Psi(t_0)}(T^ *(I\times Q)) ~|~(dH)(v)=0 \}:\]
	\[
	\omega (\dot \Psi(t_0),u)=\omega (X_H(\Psi(t_0)), u)=- (dH)(u)=0.
	\]
	If $\mathrm{im}(\Psi)$ (with $\Psi(t)=(\mathcal H(t,\psi(t)),t,\psi(t))~\forall t$) is a vortex line of $\mathcal H$ in $T^ *(I\times Q)$  the equality $\omega(\dot \Psi(t_0),u)=0~\forall u\in \mathrm{ker}(dH)_{\Psi(t_0)}=T_{\Psi(t_0)}W$ shows that $\dot \Psi(t_0)$ is proportional to $X_H(\Psi(t_0))$. Since in standard coordinates the $\frac{\partial}{\partial x}|_{\Psi(t_0)}$-component of both tangent vectors is one, it follows that $\Psi$ solves the Hamiltonian equation for $H$, i.e. the third assertion implies the second. \\
	
	In order to show the equivalence of the assertions 3 and 4 let $X$ be a vector field on $T^*Q$ viewed as a vector field on $I\times T^ *Q$. Then for $t_0\in I$ we have 
	\begin{align*}
	( (id_I,\psi)^*(\iota_X\omega_h))_{t_0} \left(\frac{\partial}{\partial t}\bigg|_{t_0}\right)&=\omega_h\left(X(\psi(t_0)), (id_I,\psi)_*\left(\frac{\partial}{\partial t}\bigg|_{t_0}\right)\right)\\
	&=\omega^{T^ *(I\times Q)} \left(h_*(X(\psi(t_0)) ), h_*\circ (id_I,\psi)_* \left(\frac{\partial}{\partial t}\bigg|_{t_0}\right)\right)\\&=\omega^{T^ *(I\times Q)} \Bigg(h_*(X(\psi(t_0)) ), \underbrace{\Psi_*\left(\frac{\partial}{\partial t}\bigg|_{t_0}\right)}_{\dot \Psi(t_0)}\Bigg).
\end{align*}
Let us observe that $T_{(t_0,\psi(t_0))}(I\times T^*Q)$ is generated by the subspace \newline $V:= \{X(\psi(t_0))~ |~X ~\text{ vector field on }T^*Q\}$ and the tangent vector $\frac{\partial}{\partial x}|_{t_0}$ as well as by $V$ and the derivative of the curve $t\mapsto (t,\psi(t))$ at time $t_0$. Since $h$ is a diffeomorphism from $I\times T^* Q$ to $W=\{H=0\}$,  $T_{\Psi(t_0)}W$ is generated by $h_*(V)$ and $h_*(\frac{d}{dt}|_{t_0}(t,\psi(t)))=\dot \Psi (t_0)$. Given that  $\omega^ {T^*(I\times Q)}(\dot \Psi(t_0), \dot{\Psi}(t_0))=0$ we arrive at the conclusion that $\dot \Psi (t_0)$ generates $\mathrm{ker}(\omega^ {T^*(I\times Q)}|_{T_{\Psi(t_0)}W})$ if and only if $(id_I,\psi)^ *(\iota_X\omega_h)=0$ for all vector fields on $T^ *Q$. (Obviously it is enough to consider vector fields with compact support.) The equivalence of assumptions 3 and 4 is thus shown.\\

Before showing the equivalence of the last two assertions let us recall what a ``variation of a section'' in the given situation is: Let $X$ be a vector field (with compact support) on $T^*Q$, considered as a vertical field on $I\times T^* Q$ over $I$. Integrating $X$ yields a flow $(\sigma_\epsilon^X)_{\epsilon\in \mathbb R}$ on $T^* Q$ (resp. $(id_I\times \sigma_\epsilon^X)$ on $I\times T^*Q$) and $(id_I,\psi)$ is a critical section if and only if 
\[
0=\frac{d}{d\epsilon}\bigg|_0\mathbb H[(id_I, \sigma_\epsilon^X\circ \psi)]~~~~~~\forall X\in \mathfrak X(T^*Q) \text{ with compact support}.
\]
Using the fundamental theorem of differential and integral calculus we obtain
\[
\frac{d}{d\epsilon}\bigg|_0 \mathbb H[(id_I, \sigma_\epsilon^X\circ \psi)]
=\int_I (id_I,\psi)^* \frac{d}{d\epsilon}\bigg |_0(id_I\times \sigma_\epsilon^X)^*\theta_h=\int_I (id_I,\psi)^*(\mathcal L_X\theta_h) 
\]
\[
=-\int_I (id_I,\psi)^*(\iota_X\omega_h) + \int_I(id_I,\psi)^*d(\iota_X\theta_h)=-\int_I(id_I,\psi)^*(\iota_X\omega_h)
,\]
since we can assume that $X$ vanishes near the boundary of $ \mathrm{im}((id_I,\psi))$. It follows that $(id_I, \psi)$ is a critical section if and only if $(id_I,\psi)^*(\iota_X\omega_h)=0$ for all $X\in \mathfrak (T^*Q)$ with compact support, i.e. statements 4 and 5 are equivalent.
\end{proof}

\begin{remark}$ $
	\begin{enumerate}
		\item The history of the equivalence theorem is long, compare e.g. \cite{MR997295} or \cite{MR0334772} for crucial points. Kijowski calls the ``new'' equations in statement 2 relative to 1, ``energy equations''.
		\item The spaces $I\times T^*Q$ resp. $T^ *(I\times Q)=\mathbb R\times (I\times T^*Q)$ are sometimes called the ``(simply) extended phase space'' resp. the ``doubly extended phase space''.
	\end{enumerate}
\end{remark}

\subsection{Jet bundles and their duals}\label{subsec1-2}
The multiphase spaces generalizing cotangent bundles (often called ``phase spaces'') are here described in terms of jet bundles and their duals as well as in terms of restricted multicotangent bundles, both points of view being useful in Subsection \ref{subsec1-3}.
\begin{definition}[First jet bundle.]
	Let $\Sigma$ resp. $Q$ be manifolds with local coordinates is $(x^1,..,x^n)$ resp. $(q^1,...,q^N)$, where $n, N\geq 1$, and $E\overset{\pi}{\to} \Sigma$ a fiber bundle with typical fiber $Q$.
	\begin{enumerate}
		\item Given $x\in\Sigma$ amd $\phi_1,\phi_2$ local sections of $\pi$ near $x$, we call $\phi_1$ and $\phi_2$ ``1-equivalent at $x$'' if $\phi_1(x)=\phi_2(x)$ and in local coordinates $(x^\mu)$ near $x$ and $(x^\mu,q^a)$ near $\phi_1(x)=\phi_2(x)$ we have $\forall \mu,\forall a$:
		\[
		\frac{\partial(q^a\circ \phi_1)}{\partial x^\mu}(x)=	\frac{\partial(q^a\circ \phi_2)}{\partial x^\mu}(x).
		\]
		The 1-equivalence class at $x$ of a local section of $\phi$ of $\pi$ near $x$ is denoted by $j_x^1(\phi)$.
		\item We denote by $J^1\pi$ the ``first jet bundle of $\pi$'', set-theoretically defined as 
		\[
		\bigsqcup_{x\in \Sigma} \{~ j_x^1(\phi)~|~\phi \text{ a local section of $\pi$ near} ~x \},
		\]
		and its projections to $E$ resp. $\Sigma$ by $\pi_{1,0}$ resp. $\pi_1=\pi\circ \pi_{1,0}$. Explicitly, one has 
		\[
		\pi_{1,0}(j_x^1(\phi))=\phi(x) \text{ and } \pi_1(j_x^1(\phi))=x.
		\]
		\item Given coordinates $(x^\mu,q^a)$ on an open set $O\subseteq E$ as above, the induced coordinates on $O^1:=\{j_x^1(\phi)~| ~\phi(x)\in O\}\subseteq J^1\pi$ are defined as $(x^\mu, q^a,v_\mu^a)$ with $x^\mu(j_x^1(\phi))=x^\mu(x)$, $q^a(j_x^1(\phi))=q^a(\phi(x))$ and $v_\mu^a(j_x^1(\phi))=\frac{\partial (q^a\circ \phi)}{\partial x^\mu}(x)$. (Attention to the abuse of notation: $x$ denotes a point in $\Sigma$ and $(x^1,...,x^n)(x)$ its local coordinates!)
	\end{enumerate}
\end{definition}

\begin{remark}$ $
	\begin{enumerate}
		\item See, e.g., \cite{MR989588} for a detailed exposition of jet bundles.
		\item Note that $j_x^1(\phi_1)=j_x^1(\phi_2)$ if and only if $T_x\phi_1=T_x\phi_2$ as a map from $T_x\Sigma$ to $T_yE$, with $\phi_1(x)=y=\phi_2(x)$.
		\item We will always use coordinates on $E$ coming from coordinates $(x^\mu)$ on an open set $U\subseteq \Sigma$ such that $\pi:\pi^{-1}(U)\to U$ is trivializable and $(q^a)$ on an open set $V$ of $Q$ such that, after trivializing $\pi$ over $U$, 
		$U\times V\subset \pi^{-1}(U)\cong U\times Q$.
		\item If $E=\Sigma \times Q$ is a product and $\pi=\mathrm{proj}_{\Sigma}$, $J^1\pi \overset{\pi_{1,0}}{\to}{E}$ is a vector bundle, canonically isomorphic to $\pi^*(T^*\Sigma)\otimes V(\pi)\to E$, where $V(\pi)=\mathrm{ker}(\pi_*)\subset TE$ is the vertical subbundle of $TE$. Observe that in the product situation $V(\pi)=(\mathrm{proj}_Q)^*(TQ)$. In general, $J^1\pi\overset{\pi_{1,0}}{\to}E$ is only an affine bundle, modelled on the vector bundle $\pi^ *(T^*\Sigma)\otimes V(\pi)\to E$. This can be easily seen upon considering two sections $\phi_1,\phi_2$ with $\phi_1(x)=y=\phi_2(x)$. Then $T_x\phi_1-T_x\phi_2=A:T_x\Sigma\to V_y(\pi)=\mathrm{ker}(\pi_*)_y\subset T_yE$ is a linear map, i.e. an element of $(\pi^ *(T^*\Sigma)\otimes V(\pi))_y$. It follows that $E_y$ can be identified with the affine spaces of linear sections of $$0\to V_y(\pi)\to T_yE\overset{(\pi_*)_y}{\to}T_x\Sigma\to 0.$$
		\item If $E=\Sigma\times \mathbb R\overset{\pi}{\to}\Sigma$, one has $J^1\pi=\pi^*(T^ *\Sigma)\otimes (\text{proj}_\mathbb R)^*(T\mathbb R)=\pi^*(T^*\Sigma)=T^ *\Sigma\times \mathbb R$, and for a smooth function $\tilde \phi:\Sigma \to \mathbb R$ with associated section $\phi=(id_I,\tilde \phi)$ of $\pi$, one has $j_x^1(\phi)=((d\phi)_x,\phi(x))$.
	\item In case $\Sigma=I\subset \mathbb R$ is an open interval and $E=I\times Q$ , we have a canonical identification 
	\begin{align*}
J^1\pi=\pi^*(T^*\Sigma)\otimes V(\pi)=( (I\times Q)\times \mathbb R)\otimes (\mathrm{proj}_Q)^*(TQ)&\\=(\mathrm{proj}_Q)^*(TQ)=I\times TQ\to I\times Q&=E.
		\end{align*} 
	Given $\tilde \phi :I\to Q$ and $\phi:=(id_I,\tilde \phi)$, $j_x^1(\phi)$ is identified with $(t,\dot{\tilde{\phi}}(t))\in I\times T_{\tilde \phi(t)}Q\subset I\times TQ$.
	\end{enumerate}
\end{remark}

\subsubsection*{Recap' on affine spaces}
Let $\mathbb K$ be a field of characteristic zero and $V,W$ two $\mathbb K$-vector spaces of finite dimension and $A,B$ affine spaces modelled on $V$ resp. $W$. Then the space of affine maps from $A$ to $B$, $\mathrm{Aff}(A,B)$ is again an affine space, modelled on $W\oplus \mathrm{Hom}_{\mathbb K}(V,W)$. If $B=W$ is a vector space, $\mathrm{Aff}(A,W)$ is a vector space as well and $W\subset \mathrm{Aff}(A,W)$ as constant maps.\\

Given a linear subspace $Z$ of the model space $V$ of an affine space $A$, one can define the quotient space $A/Z$, again an affine space and modelled on $V/Z$. We call the vector space $A^*=\mathrm{Aff}(A,\mathbb K)/\mathbb K$ the ``affine dual of $A$''. More generally, if $D$ is a one-dimensional $\mathbb K$-vector space, we call $\mathrm{Aff}(A,D)/D$ the ``$D$-twisted dual of $A$'', isomorphic to its model $\mathrm{Hom}_\mathbb K(V,D)$.

\begin{lemma}
	Let $W\overset{\pi}{\to}U$ be a surjective linear map of finite-dimensional $\mathbb K$-vector spaces with kernel $V\subset W$. Then 
	\begin{enumerate}
	\item The set $S:=\{ \sigma:U\to W \text{ $\mathbb K$-linear}~|~ \pi\circ \sigma=id_U \}$ is an affine space modelled on $\mathrm{Hom}_\mathbb K(U,V)$, called the ``space of sections of $\pi$''.
	\item For $0\leq k \leq n:=\mathrm{dim}_\mathbb KU$ let $\Lambda^n_kW^*$ be defined as \[\{\eta\in \Lambda^nW^*~|~\forall v_1,..,v_{k+1}\in V, \iota_{v_{k+1}}...\iota_{v_1}\eta=0 \}.\] Then $\Lambda^n_kW^*$ is a linear subspace of $\Lambda^nW^*$ and $\Lambda^n_0W^*=\pi^ *(\Lambda^nU^*)$, $\Lambda^n_nW^*=\Lambda^nW^*$.
	\item The vector space $\mathrm{Aff}(S,\Lambda^nU^*)$ is canonically isomorphic to $\Lambda^n_1W^*$, the isomorphism sending the constant maps $\Lambda^nU^*$ to $\Lambda^n_0W^*$.
	\item The $\Lambda^nU^*$-twisted dual of $S$ is canonically isomorphic to $\Lambda^n_1W^*/\Lambda^n_0W^*$.
	\end{enumerate}
\end{lemma}
\begin{proof}
	The main point is the construction of a map from $\Lambda^n_1W^*$ to $\mathrm{Aff}(S,\Lambda^nU^*)$. Define, for $\eta\in \Lambda^nW^*$, $\hat \eta:S\to \Lambda^nU^*$ by $\hat \eta(\sigma)=\sigma^*(\eta)$. In order to check when $\hat \eta$ is affine we choose an ordered basis $\{e_1,...,e_n\}$ of $U$ and with $\mathrm{vol}^U:=e_1^*\wedge ...\wedge e_n^*$, we obtain a map $\tilde \eta :S\to \mathbb K$ by $\hat \eta (\sigma)=\sigma^*(\eta)=\tilde \eta(\sigma)\cdot \mathrm{vol}^U$.
	Since $\hat \eta$ is affine if and only if $\tilde \eta$ is affine, we can take $\lambda\in \mathrm{Hom}_{\mathbb K}(U,V)$, the linear model space of $S$ and check that 
	\[\tilde \eta(\sigma+\lambda)-\tilde \eta(\sigma)=\sum_{j=1}^n\eta(\sigma(e_1),...,\lambda (e_j),...,\sigma (e_n)) + \text{`` terms with two or more $\lambda$'s ''}.
	\]
	It follows that $\tilde \eta$ (and thus $\hat \eta$) is affine if and only if $\eta\in \Lambda^n_1W^*$. More details on this construction can be found, e.g., in \cite{MR1244450}.
\end{proof}
\begin{remark}
	If we fix a volume form on $U$, $\Lambda^nU^*$ and $\Lambda^n_0W^*=\pi ^*(\Lambda^nU^*)$ are canonically identified with $\mathbb K$.
\end{remark}

\begin{construction}
	Let $E\overset{\pi}{\to}\Sigma$ a smooth fiber bundle with typical fiber $Q$, and let $n=\mathrm{dim}_\mathbb R\Sigma$, $N=\mathrm{dim}_\mathbb RQ$, as well as 
	\[
	    \xymatrix{
	    	J^1\pi \ar[d]_{\pi_{1,0}} \ar@/^2pc/[dd]^\pi\\
	    	E \ar[d]_{\pi}\\
	    	M}
	\]
	the first jet bundle of $\pi$.  
	Applying fiberwise the constructions of the preceding lemma, we obtain the $\pi^*(\Lambda^nT^*\Sigma)$-twisted dual $P(\pi)$ of $J^1\pi$ and the bundle $M(\pi)$ of fiberwise (over $E$) affine maps from $J^1\pi$ to $\pi^*\Lambda^nT^*\Sigma$. Both spaces are vector bundles over $E$ and ``identified'' by the next proposition. 
\end{construction}

\begin{definition}
	Let $\pi:E\to \Sigma$ be a fiber bundle with typical fiber $Q$ and let $n=dim_\mathbb R\Sigma$. For $0\leq k\leq n$ we set 
	\begin{align*}
	\Lambda^n_kT^*E=\{  \eta\in \Lambda^nT^*E~|~ \eta \text{ lies over } y\in E \text{ and } \forall u_1,..,u_{k+1}\in V_y(\pi)=\mathrm{ker}(\pi_*)_y,\\ \iota_{u_{k+1}}...\iota_{u_1}\eta=0 \}.
		\end{align*}	
\end{definition}

\begin{remark}
	For $0\leq k\leq n$, $\Lambda^n_kT^*E$ is a sub vector bundle of $\Lambda^nT^*E$ and $\Lambda^n_0T^*E$ is canonically isomorphic to $\pi^*(\Lambda^ n T^*\Sigma)$, whereas $\Lambda^n_nT^*E=\Lambda^nT^ *E$.
\end{remark}

\begin{proposition}\label{propj}
	Let $\pi:E\to \Sigma$ be a fiber bundle with typical fiber $Q$ and let $n=\mathrm{dim}_\mathbb R\Sigma$. Then
	\begin{enumerate}
		\item There is an isomorphism  $M(\pi)\cong \Lambda^n_1T^*E$ of vector bundles over $E$ and the latter has coordinates $(x^\mu, q^a, p^\mu_a,p)$, where $(x^1,...,x^n)$ are local coordinates on $\Sigma$, $(q^1,...,q^N)$ are local coordinates on $Q$ and an element $ \eta\in (\Lambda^n_1T^ *E)_{x,q}$ is given by $\eta=pdx^1\wedge ...\wedge dx^n + \sum_{\mu,a} p_a^\mu dq^a\wedge \widehat{d^nx^\mu}$, where $\widehat{d^nx^\mu}
		=\iota_{\frac{\partial}{\partial x^\mu}}(dx^1\wedge ...\wedge dx^n)
		=(-1)^{\mu+1}dx^1\wedge ... \wedge dx^{\mu-1}\wedge dx^{\mu+1}\wedge ...\wedge dx^n$.
		\item The following diagram commutes
			\[
			\xymatrix{
			\Lambda^nT^ *E	&\ar[l]_\supset\Lambda^n_1T^*E\ar[r]^\cong \ar[dd]&M(\pi)\ar[rd] \ar[dd]^\mu&\\
				&&&E\\
				V^*(\pi)\otimes \pi^*(\Lambda^{n-1}T^ *\Sigma)\ar[r]^{~~~~~\cong}&\Lambda^n_1T^*E/\Lambda^n_0T^*E\ar[r]^{~~~~\cong}&P(\pi)\ar[ru]&
				},
			\]
		where the oblique arrows are projections of vector bundles and the vertical arrows are surjective submersions, realising $P(\pi)$ as a quotient bundle of $\Lambda^n_1T^*E\to E$. Furthermore, $P(\pi)$ is isomorphic to $\pi^*(T\Sigma)\otimes V^*(\pi)\otimes \pi^ *(\Lambda^nT^*\Sigma)=V^*(\pi)\otimes \pi^*(\Lambda^{n-1}T^ *\Sigma)$.
		The manifold $P(\pi)$ has local coordinates $(x^\mu,q^a,p_a^\mu)$ analogous to the local coordinates of $M(\pi)$ given by point 1.
	\end{enumerate}
\end{proposition}

\begin{proof}
	Follows from the preceding linear algebra. For the last identification of $P(\pi)$ we use that, given an $n$-dimensional $\mathbb K$-vector space $V$, the map $V\otimes \Lambda^nV^*\to \Lambda^{n-1}V^*$, $v\otimes \mu\mapsto \iota_v\mu$ is a linear isomorphism.
\end{proof}

\begin{remark}$ $
	\begin{enumerate}
		\item The manifolds associated to the fibration $E\overset{\pi}{\to}\Sigma$ can be resumed by the following diagram of fibrations
		\[
		\xymatrix{
				\Lambda^n_1T^ *E \ar[d] \ar[r]^\cong&M(\pi) \ar[d]_{\mu}&\\
				\Lambda^n_1T^*E/\Lambda^n_0T^*E \ar[r]^{~~~~~\cong}&P(\pi) \ar@/_2pc/[dd]_{\tau} \ar[d]_{\kappa} &J^1\pi \ar[dl]_{\pi_{1,0}} \ar[ddl]^{\pi_1}\\
				&E \ar[d]^Q_\pi&\\
				&\Sigma&
			}		
		\]
		We put here, for later use, $\tau:=\pi\circ\kappa$.
		\item In case $\Sigma=I\subseteq \mathbb R$ is an interval and $E=I\times Q$ is trivialized, we obtain the following canonical identifications: $J^1\pi=I\times TQ$, $P(\pi)=I\times T^ *Q$, $M(\pi)=T^ *(I\times Q)$, i.e. we are in the situation of time-dependent classical mechanics revisited in Subsection \ref{subsec1}. For this analogy, the space $P(\pi)$ is sometimes called the ``(simply) extended multiphase space'' and $M(\pi)$ the ``doubly extended multiphase space''. Further choices of standard ``physical'' terminology include ``restricted (resp. extended) multimomentum bundle'' for $P(\pi)$ resp. $M(\pi)$.
		\item More generally, we can identify all manifolds in 1. in the case $E=\Sigma\times Q$: \[J^1\pi\cong \pi^*(T^ *\Sigma)\otimes (\mathrm{proj}_Q)^*TQ=T^*\Sigma\otimes TQ\to E\] (suppressing pullbacks for shortness here). Furthermore \[P(\pi)=V^ *(\pi)\otimes \pi^ *(\Lambda^{n-1}T^ *\Sigma)=T^*Q\otimes \Lambda^{n-1}T^*\Sigma\to E \]  and finally 
		\[M(\pi)=\Lambda^nT^*\Sigma \oplus (T^ *Q\otimes \Lambda^{n-1}T^*\Sigma) \to E,\] projecting as a vector bundle over $E$ in the obvious way onto $P(\pi)$.
	\end{enumerate}
\end{remark}

\begin{remark}\label{lemmj}The fibration $\mu:M(\pi)\to P(\pi)$ is always an affine real line bundle with associated (linear) real line bundle $\tau^*(\Lambda^n T^ *\Sigma)$.
\end{remark}

\subsection{Hamiltonian approach to classical field theories}\label{subsec1-3}
We  begin by explaining the transition from the Lagrangian to the Hamiltonian formulation of classical field theories. Then we give -in a purely Hamiltonian setting- various equivalent formulations of the condition that a section of a bundle $\pi:E\overset{Q}{\to} \Sigma$ (a ``field'' in physical lingo) is a solution of a given classical field theory. Our efforts culminate in the second condition of Theorem \ref{fin1thm}, that allows to formulate Hamiltonian dynamics on arbitrary multisymplectic manifolds at the end of Section \ref{sec:2}.\\

The analogues of the canonical 1- and 2-form on a cotangent bundle are described by the following easy but fundamental
\begin{propdef}[Tautological forms on multicotangent bundles]~
	\begin{enumerate}
		\item Let $Y$ be a smooth manifold of dimension at least $n\geq 1$. Then the ``multicotangent bundle (or multimomentum space)'' $\Lambda^nT^ *Y$ carries a ``tautological or canonical $n$-form'' $\theta^{\Lambda^nT^ *Y}$ defined by 
		\begin{align*}
		\theta_{\eta_y}^{\Lambda^nT^ *Y} (u_1,...,u_n)=\eta_y( (\mathrm{proj}_Y)_*(u_1),...,(\mathrm{proj}_Y)_*(u_n))\\\forall y\in Y,\forall \eta_y\in \Lambda^nT^*Y, \forall  u_1,...,u_n\in T_{\eta_y}(\Lambda^nT^ *Y),
		\end{align*}
		and a ``canonical $(n{+}1)$-form'' $\omega^{\Lambda^nT^ *Y}:=-d\theta^ {\Lambda^nT ^ *Y}$. If $y^1,...,y^N$ are local coordinates  on $U\subset Y$, then $(p_I, y^i)$ are the local coordinates describing $\sum_Ip_Idy^I$ in $\Lambda^nT^*Y$, where $I=(i_1,...,i_n)$ are strictly ascending multiindices and $dy^I=dy^{i_1}\wedge ...\wedge dy^{i_n}$. With respect to these coordinates we have $\theta^{\Lambda^nT^*Y}=\sum_Ip_Idy^I$, and consequently $\omega^{\Lambda^nT^*Y}=-\sum_Idp_I\wedge dy^I$. 
	
		The form $\omega$ is ``non-degenerate'', i.e. $\forall y\in Y, \forall\eta_y \in \Lambda^nT^*_yY$, 
		\[(\omega_{\eta_y}^{\Lambda^nT^ *Y})^\#:T_{\eta_y}(\Lambda^nT^*Y)\to \Lambda^n T^*_{\eta_y}(\Lambda^nT^*Y), \]
		given by the contraction of a tangent vector with the $(n{+}1)$-form $\omega_{\eta_y}^{\Lambda^nT^*Y}$, is injective.
		\item Let $Y=E\overset{\pi}{\to}\Sigma$ be a fibration over the $n$-dimensional manifold $\Sigma$ with typical fiber $Q$ and $1\leq k\leq n$. Then the pullback of $\omega^{\Lambda^nT^*E}$ to $\Lambda^n_kT^ *E$ is non-degenerate as well.
	\end{enumerate}
\end{propdef}

\begin{proof}
	To see the non-degeneracy of $\omega^{\Lambda^nT^*Y}$ let $v_0=\sum_ia_i\frac{\partial}{\partial y^i}+ \sum_Ib_I\frac{\partial}{\partial p_I}$ be a non-zero tangent vector to $Y$. If there exists an $I=(i_1,...,i_n)$ such that $b_I\neq 0$, then $v_j=\frac{\partial}{\partial y_{i_j}}$ for $j\in\{1,...,n\}$ satisfy $\omega^{\Lambda^nT^*Y}(v_0,v_1,...,v_n)=b_I\neq 0$. If $b_I$ is zero for all $I$, then there is at least one $i$ such that $a_i\neq 0$. Without loss of generality we assume $i=1$. Then for $v_1=\frac{\partial}{\partial p_{(1,...,n)}}$, and $v_j=\frac{\partial}{\partial y_j}$ for $j\in\{2,...,n\}$ we have $\omega^{\Lambda^nT^*Y}(v_0,v_1,...,v_n)=-a_1\neq 0$. Hence $\iota_{v_0}\omega\neq 0$ for all nonzero $v_0$, i.e. $\omega^{\Lambda^nT^*Y}$ is non-degenerate. For the subbundles we can choose a chart of $Y$ such that $y_1,..,y_n$ is a chart of $\Sigma$. Then $\Lambda^n_kT^*U$ has coordinates $(p_I,y^i)$, where now $I$ runs through all multi-indices which contain at most $k$ of the elements $\{1,...,n\}$. The conditions on $k$ guarantee, that such multiindices exist. Again,  $\omega^{\Lambda^nT^*Y}|_{\Lambda_k^nT^*Y}=-\sum_Idp_I\wedge dy^I$ with the new (restricted) index subset $I$, and non-degeneracy can be shown as for $\omega^{\Lambda^nT^*Y}$.
\end{proof}

\begin{corollary}
	Let $\pi:E\overset{Q}{\to}\Sigma$ be as in part 2 of the preceding definition, $(x^1,...,x^n)$ local coordinates on $\Sigma$, $(q^a)$ local coordinates on $Q$ and for $y\in E$, elements of $\Lambda^n_1T^*_yE$ written as follows 
	\[
	\eta_y=p(d^nx)_y+\sum_{a}(dq^a)_y\wedge (\sum_\mu p^\mu_a(\widehat{d^nx^\mu})_y),
	\]
	where $d^nx$ and $\widehat{d^nx^\mu}$ are as in Proposition \ref{propj} above. Then, near $y$, $\theta:=\theta^{\Lambda^n_1T^*E}=pd^nx + \sum_{a,\mu}p^\mu_a dq^a\wedge \widehat{d^nx^\mu}$ and $\omega:=\omega^{\Lambda^n_1T^*E}=-dp\wedge d^nx - \sum_{a,\mu}dp^\mu_a\wedge dq^a\wedge \widehat{d^nx^\mu}$.
\end{corollary}

Let us from now to the end of this section fix a $Q$-fiber bundle $\pi:E\to \Sigma$ with $n$-dimensional base. Recall that $M(\pi)=\mathrm{Aff}(J^1\pi,\pi^*\Lambda^nT^ *\Sigma)$ is isomorphic to $\Lambda^n_1T^ *E$ and $P(\pi) =\mathrm{Aff}(J^1\pi,$ $\pi^ *\Lambda^n T^ *\Sigma)/\pi^ *(\Lambda^nT^ *\Sigma)$ is isomorphic to $\Lambda^n_1T^ *E/\Lambda^n_0T^ *E$, the bundle $M(\pi)\overset{\mu}{\to}P(\pi)$ being an affine real line bundle (cf. Remark \ref{lemmj}).\\

Before discussing classical field theory in the Hamiltonian approach, let us rapidly review the more standard Lagrangian approach.

\begin{definition}
Let $\mathcal L:J^1\pi\to \pi^*\Lambda^nT^ *\Sigma$ be a smooth map over $E$. We call $\mathcal L$ a ``Lagrangian density''.
\begin{enumerate}
	\item The ``Legendre transformation (associated to $\mathcal L$)'' is the smooth map $\mathbb F\mathcal L:J^1\pi\to \mathrm{Aff}(J^1\pi,$ $\pi^ *\Lambda^nT^ *\Sigma)$ given by $((\mathbb F\mathcal{L})(v))(w)=\mathcal{L}(v)+\frac{d}{d\epsilon}\big|_0\mathcal L(v+\epsilon(w-v))$, where $v,w\in (J^1\pi)_y$ for $y\in E$ and $w-v\in \mathrm{Hom}_\mathbb R(T_{\pi(y)}\Sigma, V_y(\pi))$, the vector space model of the affine space $(J^1\pi)_y$.
	\item We call a Lagrangian density $\mathcal L$ ``regular resp. hyper-regular'' if the map $\mu\circ \mathbb F\mathcal L:J^1\pi\to P(\pi)$ is a local diffeomorphism resp. a diffeomorphism. In these cases we call $h:=(\mathbb F \mathcal L)\circ (\mu \circ \mathbb F\mathcal L)^{-1} $ the ``(local) Hamiltonian section associated to $\mathcal L$''.
\end{enumerate}	
\end{definition}

\begin{remark}$ $
	\begin{enumerate}
		\item Often the map $\mu \circ \mathbb F\mathcal L$ is called the Legendre transformation associated to $\mathcal L$. We stick to the convention that $\mathbb F\mathcal L$ is the Legendre transformation.
		\item In case $\mu$ is trivialized, $h$ is an $\mathbb R$-valued function. This is the typical case in the classical field theories considered in physics (compare Remark \ref{voltref} below). 
		\item In local coordinates on $J^1\pi$ and $M(\pi)=\mathrm{Aff}(J^1\pi,\pi^ *\Lambda^n T^ *\Sigma)$ we can assume that $\mathcal L=Ldx^1\wedge ...\wedge dx^n$ and we have $(x^\mu, q^a,p^\mu_a,p)(x^\mu,q^a,v^a_\mu)=(x^\mu,q^a,p+\sum_{a,\mu}p^\mu_av^a_\mu)$, if we consider $(x^\mu, q^a,p^\mu_a,p)$ as an affine map. Furthermore,\newline $(\mathbb F\mathcal L)(x^\mu,q^a,v^a_\mu)=(x^\mu,q^a,\frac{\partial L}{\partial v^a_\mu},L-\sum_{a,\mu}\frac{\partial L}{\partial v^a_\mu}\cdot v^a_\mu)$ and $\mathcal{ L}$ is regular if $p^\mu_a=\frac{\partial L}{\partial v^a_\mu}(x^\nu,q^b,v^b_\nu)$ is locally solvable to obtain $v^b_\nu=G_\nu^b(x^\mu,q^a,p^\mu_a)$. (Observe that for $n=1,\mu=1$ and we have found the standard expression for $H$ in terms of $L$, well-known from classical mechanics $H=L-\sum_a \frac{\partial L}{\partial v^a}v^a$.)
	\end{enumerate}
\end{remark}

\begin{lemma}
	Let $\mathcal L$ be a Lagrangian density and $\phi$ a section of $\pi:E\to \Sigma$, then $\mathcal L\circ j^1\phi=(j^1\phi)^*(\theta_\mathcal L)$, where $\theta_\mathcal L=(\mathbb F\mathcal{L})^*\theta$, and the following are equivalent:
	\begin{enumerate}
		\item $\phi$ is critical for the functional $\mathbb L[\phi]:=\int_{\Sigma} \mathcal L(j^1\phi)$ on sections of $E\overset{\pi}{\to}\Sigma$,
		\item in local coordinates $\phi$ satisfies the following ``Euler-Lagrange equations'':
		\[
		\forall a\in \{1,...,N\},~~ \frac{\partial L}{\partial q^a}((j^1\phi)(x))=\sum_{\mu} \frac{\partial}{\partial x^\mu}\left( \frac{\partial L}{\partial v^a_\mu}\circ (j^1\phi)(x) \right).
		\]
	\end{enumerate}
\end{lemma}
\begin{proof}
	Using Proposition \ref{propj} and the above remarks we obtain in local coordinates
	$
	\theta=p\cdot d^nx+\sum_{a,\mu}p^\mu_adq^a\wedge \widehat{d^nx^\mu},
	$, $L=\mathcal L d^nx$ and $$(\mathbb F\mathcal L\circ j^1\phi)(x)=\left(x^\mu, \phi^a(x),\frac{\partial L}{\partial v^a_\mu}((j^1\phi)(x)),
	L((j^1\phi)(x))-\sum_{a,\mu}\frac{\partial L}{\partial v^a_\mu}((j^1\phi)(x))\cdot \frac{\partial \phi^a}{\partial x^\mu}(x)\right).$$
	Thus 
	\begin{align*}
	(j^1\phi)^*\theta_\mathcal L=(\mathbb F\mathcal L\circ j^1\phi)^ *\theta = L((j^1\phi)(x))d^nx-\sum_{a,\mu}\frac{\partial L}{\partial v^a_\mu}((j^1\phi)(x))\frac{\partial \phi^a}{\partial x^\mu}(x)d^nx\\+\sum_{a,\mu}\frac{\partial L}{\partial v^a_\mu}((j^1\phi)(x))\frac{\partial \phi^a}{\partial x^\mu}(x)dx^\mu\wedge \widehat{d^nx^\mu}=L((j^1\phi)(x))d^nx=(\mathcal{L}\circ j^1\phi)(x).
	\end{align*}
	Let now $X$ be a vertical vector field with compact support on $E$, in local coordinates $X=\sum_aX_a\frac{\partial }{\partial q^a}$ with $X_a=X_a(x^\mu, q^a)$, and let $(\sigma_\epsilon^X)_{\epsilon\in \mathbb R}$ be its flow on $E$. It is well-known how to lift (``prolong'') $\pi$-vertical vector fields on $E$ to vertical vector fields of $ \pi_1=\pi\circ \pi_{1,0}:J^1\pi\to \Sigma$ on the total space $J^1\pi$ (cf., e.g., \cite[Section 4.4]{MR989588}). In local coordinates $X$ prolongs to $\Sigma_aX_a\frac{\partial}{\partial q^a}+\sum_{a,\mu}\frac{\partial X_a}{\partial x^\mu}\frac{\partial }{\partial v^a_\mu}$. It follows that 
	\begin{align*}
	\frac{d}{d\epsilon}\bigg |_0\mathbb L[\sigma^X_\epsilon\circ \phi]
	&=\frac{d}{d\epsilon}\bigg |_0\int_{\Sigma}L((j^1\phi)(x))d^nx\\
	&=\int_{\Sigma}\frac{d}{d\epsilon}\bigg |_0L( x,(\sigma_\epsilon^X\circ \phi)(x),\frac{\partial}{\partial x^\mu}(\sigma^X_\epsilon\circ\phi)^a(x) )d^nx\\
	&=\int_ \Sigma \left(\sum_a \frac{\partial L}{\partial q^a}(\phi(x))X_a+\sum_{\mu,a}\frac{\partial L}{\partial v^a_\mu}( (j^1\phi) (x))\frac{\partial X_a}{\partial x^\mu}\right)d^nx\\
	&=\int_ \Sigma \left(
	\sum_a 
	X_a\left(
	\frac{\partial L}{\partial q^a}-\sum_{\mu}\frac{\partial}{\partial x^\mu}\frac{\partial L}{\partial v^a_\mu}
	\right)\right)d^nx
	\end{align*}
by partial integration. It follows that 
$\frac{d}{d\epsilon}\big|_0\mathbb L[\sigma^X_\epsilon\circ \phi]=0,~~~\forall X\in \Gamma(E,V(\pi))$ if and only if $\forall a\in \{1,...,n\}$
\[
\left(\frac{\partial L}{\partial q^a}\right)\circ (j^1\phi)(x)-\sum_{\mu}\left( \frac{\partial L}{\partial v^a_\mu}\circ (j^1\phi)(x) \right)=0.
\]
\end{proof}

\noindent The preceding lemma allows for an important generalisation.
\begin{proposition}
	Let $\pi:E\overset{Q}{\to} \Sigma$ a fiber bundle, $\mathcal L$ a Lagrangian density on $J^1\pi$ and $s$ a section of $\pi_1:J^1\pi\to \Sigma$. Then the following are equivalent:
	\begin{enumerate}
		\item $s$ is a critical section for $\mathbb L [s]:=\int_{\Sigma} \mathcal L\circ s=\int_{\Sigma}s^*\theta_{\mathcal L}$.
		\item $s^*(\iota_X\omega_{\mathcal L})=0$ for all $X\in\mathfrak X(J^1\pi)$ that are $\pi_1$-vertical, where $\omega_\mathcal L=-d\theta_\mathcal L=(\mathbb F\mathcal L)^*\omega$.
		\item $s^*(\iota_X\omega_\mathcal L)=0$ for all $X\in \mathfrak X(J^1\pi)$.
		\item $s=j^1\phi$ with $\phi$ a section of $\pi:E\to \Sigma$ and $\phi$ is critical for $\mathbb L[\phi]=\int_{\Sigma }\mathcal L\circ j^1\phi$.
	\end{enumerate}
\end{proposition}

\begin{proof}
	See, e.g., \cite{MR0341531} or the survey \cite{MR2559661}.
\end{proof}

\begin{corollary}
	Let $\mathcal L$ be a hyper-regular Lagrangian density on $J^1\pi$, s a section of $\pi_1:J^1\pi\to \Sigma$ and $\tilde \Psi$ a section of $\tau:P(\pi)\to \Sigma$ such that $\tilde \Psi=(\mu \circ \mathbb F\mathcal L)\circ s$ (or equvalently $s=(\mu\circ \mathbb F\mathcal L)^{-1}\circ \tilde \Psi$). Then $s$ is critical for $\mathbb L[s]=\int_{\Sigma} s^*\theta_\mathcal L$ if and only if $\tilde \Psi$ is critical for the ``Hamilton functional'' $\mathbb H[\tilde \Psi]=\int _ \Sigma \tilde \Psi^ *\theta_h$, where $\theta_h=h^ *\theta$.
\end{corollary}

\begin{proof}
Recall first the relevant diagram:
			
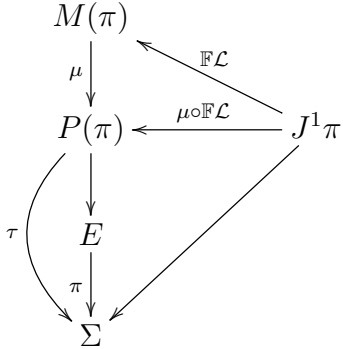
\[
			\xymatrix{
				M(\pi)\ar[d]_\mu&&\\
				P(\pi)\ar[d] \ar@/_2pc/[dd]_{\tau}&& J^1\pi \ar[ll]_{\mu\circ \mathbb F\mathcal L} \ar[ull]_{\mathbb F\mathcal L} \ar[ddll]  \\
				E\ar[d]_\pi&&\\
				\Sigma&&
			}		
			\]		
One has $\tilde \Psi^ *\theta_h=s^ *\circ (\mu\circ \mathbb F\mathcal L)^*\circ ((\mu\circ \mathbb F\mathcal L)^{-1})^*\circ (\mathbb F\mathcal L)^*\theta=s^*\circ (\mathbb F\mathcal L)^ *\theta=s^*\theta_\mathcal L$.
\end{proof}

\begin{definition}
	Let $h$ be a smooth map from $P(\pi)$ to $M(\pi)$ such that $\mu\circ h=\mathrm{id}_{P(\pi)}$. Then $h$ is called a ``Hamiltonian section (of $\mu$)'' and we denote $\mathrm{im}(h)\subset M(\pi)$ as $W$. The image $\Psi(U)$ of a (local) section $\Psi:U\to M(\pi)$ of $\tau\circ \mu:M(\pi)\to \Sigma$, defined on an open subset $U\subset \Sigma$, is called a ``(local) vortex $n$-plane (for $h$)'' if 
	\begin{enumerate}
		\item $\Psi(U)\subset W$, i.e. with $\tilde \Psi:=\mu\circ \Psi$ one has $\Psi=h\circ \tilde \Psi$, and
		\item $\forall x\in U$, $\forall \gamma_x\in \Lambda^nT_x\Sigma$ one has $\iota_{(\Psi_*)_x(\gamma_x)}\omega_{\Psi(x)}=0$ as a functional on $T_{\Psi(x)}W$.
	\end{enumerate}
\end{definition}
\begin{remark}
	Let $(x^\mu,q^a,p^\mu_a,p)$ resp. $(x^\mu,q^a,p^\mu_a)$ be standard coordinates on $M(\pi)$ resp. $P(\pi)$. Then $\mu(x^\mu,q^a,p^\mu_a,p)=(x^\mu,q^a,p^\mu_a)$ and $h(x^\mu,q^a,p^\mu_a)=(x^\mu,q^a,p^\mu_a,\mathcal H(x^\mu,q^a,p^\mu_a))$, i.e. locally 
	\[
	W=\mathrm{im}(h)=\{(x^\mu,q^a,p^\mu_a,p)~|~ \mathcal(x^\mu,q^a,p^\mu_a)=p\}
	\]
	and, putting $H(x^\mu,q^a,p^\mu_a,p)=\mathcal H(x^\mu,q^a,p^\mu_a)-p$, $W=\{H=0\}\subset M(\pi)$. Thus for a (local) vortex $n$-plane, we have $\forall x\in U \subset \Sigma$, $\iota_{(\Psi_*)_x(\gamma_x)}\omega_{\Psi(x)}$ is proportional to $(dH)_{\Psi(x)}$ for all $\gamma_x\in \Lambda^nT_x\Sigma$. We obtain then that $\omega_h=h^*\omega=-d\mathcal H\wedge d^nx-\sum_{a,\mu}dp_a^\mu\wedge dq^a\wedge \widehat{d^nx^\mu}$ in these coordinates.
	\end{remark}

\begin{theorem}[Equivalent formulations of Hamiltonian field theories, I]
	Let $\pi:E\overset{Q}{\to}\Sigma$ be a fiber bundle with $n$-dimensional base and $N$-dimensional fiber and $h:P(\pi)\to M(\pi)$ be a Hamiltonian section. Then for a (local) section $\tilde\Psi$ of $\tau:P(\pi)\to \Sigma$ defined on $U$, open in $\Sigma$, the following are equivalent:
	\begin{enumerate}
		\item $\tilde \Psi$ is critical for $\mathbb H[\tilde \Psi]=\int_U \tilde \Psi^*\theta_h$.
		\item $\tilde \Psi^ *(\iota_X\omega_h)=0$ for all $\tau$-vertical $X$ in $\mathfrak X(P(\pi))$.
		\item $\tilde \Psi^ *(\iota_X\omega_h)=0$ for all $X$ in $\mathfrak X(P(\pi))$.
		\item in local coordinates as in the preceding remark, $\tilde \Psi$ fulfills the following ``Hamilton-Volterra'' equations: $\forall \mu \in \{1,...,n\}$ and $\forall a\in \{1,...,N\}$
		\begin{align*}
			\frac{\partial \mathcal H}{\partial q^a}(\tilde \Psi(x))=\sum_{\mu=1}^{n}\frac{\partial(p^\mu_a\circ \tilde \Psi)(x)}{\partial x^\mu},\\
			-\frac{\partial \mathcal H}{\partial p_a^\mu}(\tilde \Psi(x))=\frac{\partial(q^a\circ \tilde \Psi)(x)}{\partial x^\mu}.
		\end{align*}
	 \item 	The image of $\Psi:=h\circ \tilde \Psi:U\to M(\pi)$ is a (local) vortex $n$-plane.
	\end{enumerate}
\end{theorem}

\begin{proof}
	For a proof of the equivalence of the first four conditions we refer again to \cite{MR0341531,MR2559661}, compare also \cite{MR2370237}. Here, we only show that 3. and 5. are equivalent, since the notion of a vortex $n$-plane seems to be less standard.
	Obviously $W$ together with the form $\omega|_{TW}$ is diffeomorphic to $P(\pi)$ with the form $\omega_h=h^*\omega$ via $\mu|_W$ with inverse $h$. For $X\in \mathfrak X(P(\pi))$ and $x\in U$, $\gamma_x\in \Lambda^nT_x\Sigma$ we have 
	\begin{align*}
	(\tilde \Psi ^ *(\iota_X\omega_h))_x(\gamma_x)
	&=(\omega_h)_{\tilde \Psi(x)}(X_{\tilde{\Psi(x)}}, (\tilde \Psi_*)_x(\gamma_x) )\\
	&=(-1)^n ( \iota_{(\tilde \Psi_*)_x(\gamma_x)}(\omega_h)_{\tilde \Psi(x)}  )(X_{\tilde \Psi(x)})
	\end{align*}
	and thus $\tilde \Psi ^*(\iota_X\omega_h)=0$, $\forall X\in \mathfrak X(P(\pi))$ is equivalent to 
	$ ( \iota_{ (\tilde \Psi_*)_x(\gamma_x)}\omega_{\tilde \Psi(x)}  )(v_{\Psi(x)})=0 $, $\forall x\in U\subset\Sigma$, $\forall \gamma_x\in \Lambda^nT_x\Sigma$ and $\forall v_{\Psi(x)}\in T_{\psi(x)}P(\pi)$.
\end{proof}

\begin{remark}~\label{voltref}
	\begin{enumerate}
		\item The Hamilton-Volterra equations go back, at least, to work of Volterra in the end of the 19th century (compare \cite{volt1,volt2}).
		\item When a section $h$ of the affine $\mathbb{R}$-bundle $\mu:M(\pi)\to P(\pi)$ is given, we get a linear structure on $\mu$ by the following isomorphism of affine $\mathbb R$-bundles
		\[
		\xymatrix{
		\tau^*(\Lambda^n T^*\Sigma)\ar[rr] \ar [dr]&&M(\pi)\ar [dl]_\mu\\
		&P(\pi)&			
			},
		\] 
		$\Omega_z\mapsto h(z)+\Omega_z$, $\forall z\in P(\pi)$, $\forall \Omega_z\in \tau^ *(\Lambda^nT^ *\Sigma)_z=\Lambda^nT^*_{\tau(z)}\Sigma$. If, furthermore, a volume form $\mathrm{vol}^\Sigma$ is given on $\Sigma$, the induced section $\tau^ *(\mathrm{vol}^\Sigma)$ yields a trivialization 
\[
	\xymatrix{
			P(\pi)\times \mathbb R \ar[rr] \ar [rd]&&\tau^*(\Lambda^n T^*\Sigma)\ar [dl]\\
			&P(\pi)&			
			},
\] 
$(z,u)\mapsto u\cdot (\tau^ *\mathrm{vol}^\Sigma)_z=u\cdot \mathrm{vol}^\Sigma_{\tau(z)}$. 	
Combining the two maps, we can trivialize $\mu$ by:
\[
\xymatrix{
	P(\pi)\times \mathbb R \ar[rr]^\chi \ar [rd]_{\mathrm{proj}_{P(\pi)}}&&M(\pi)\ar [dl]^\mu\\
		&P(\pi)&			
	},
	\]		
	$\chi(z,u)=h(z)+u\cdot (\tau ^ *\mathrm{vol}^\Sigma)_z$. The ``Hamiltonian section'' then translates to a ``Hamiltonian function'' $H:=-\mathrm{proj}_{\mathbb R}\circ \chi{-1}:M(\pi)\to \mathbb R$. In local standard coordinates such that $\mathrm{vol}^\Sigma=d^nx$, we get: $\chi^{-1}(x^\mu,q^a,p^\mu_a,p)=(x^\mu,q^a,p^\mu_a,p-\mathcal H(x^\mu,q^a,p^\mu_a))$ and thus $H(x^\mu,q^a,p^\mu_a,p)=\mathcal H(x^\mu,q^a,p^\mu_a)-p$.
	\end{enumerate}
\end{remark}

\begin{theorem}[Equivalent formulations of Hamiltonian field theories, II]\label{fin1thm}
	Let $\pi:E\overset{Q}{\to}\Sigma$ be a fiber bundle with $n$-dimensional base and $h:P(\pi)\to M(\pi)$ be a Hamiltonian section. Assume furthermore that a volume form $\mathrm{vol}^\Sigma$ is given on $\Sigma$ and let $\gamma\in\mathfrak X^n(\Sigma)$ be the unique $n$-vector field satisfying $\iota_\gamma\mathrm{vol}^\Sigma=1$. Then for a (local) section $\tilde \Psi$ of $\tau: P(\pi)\to \Sigma$ defined on $U$, open in $\Sigma$, the following are equivalent:
	\begin{enumerate}[(i)]
		\item $\tilde \Psi$ is a critical section for the functional $\mathbb H[\tilde \Psi]=\int _ U\tilde \Psi^ *\theta_h$, where $\theta_h=h^* \theta$ and $\theta=\theta^{M(\pi)}$,
		\item $\iota_{(\Psi_*)_x(\gamma_x)}\omega_{\Psi(x)}=(-1)^n(dH)_{\Psi(x)}$ $\forall x\in U$, where $\Psi=h\circ \tilde \Psi:U\to M(\pi)$.
	\end{enumerate}
\end{theorem}

\begin{proof}
	Let us stress, that we invert the order, when we define contractions between multivector fields and differential forms, i.e. if $V$ is a vector space of dimension at least two and $\eta\in \Lambda^kV^*$ with $2\leq k\leq dim_{\mathbb R}V$ and $u,v\in V$, then $\iota_{u\wedge v}\eta=\iota_v\iota_u\eta$.\\
	
	Recall from the first theorem on equivalent formulations of Hamiltonian field theories that the first condition is equivalent to $\Psi$ being a vortex $n$-plane, i.e. $\iota_{(\Psi_*)_x(\gamma_x)}\omega_{\Psi(x)}=0$ as functionals on $T_{\Psi(x)}W=\mathrm{ker}(dH)_{\Psi(x)}\subset T_{\Psi(x)}M(\pi)$ $\forall x\in U$. (Here $H$ is the Hamiltonian function on $M(\pi)$, associated to $h$ and $\mathrm{vol}^\Sigma$.) 
	Thus $\iota_{(\Psi_*)_x(\gamma_x)}\omega_{\Psi(x)}=g(x)(dH)_{\Psi(x)}$ as functionals on $T_{\Psi(x)}M(\pi)$, for all $x\in U$, where $g:U\to \mathbb R$ is a smooth function. Since in local coordinates $\mathrm{vol}^\Sigma=d^nx=dx^1\wedge ... \wedge dx^n$, $\gamma=\frac{\partial}{\partial x^1}\wedge ... \wedge \frac{\partial}{\partial x^n}$,
	$\omega=-dp\wedge d^nx-\sum_{ a,\mu} dp^\mu_a\wedge dq^a\wedge \widehat{d^nx^\mu}$ and $H=\mathcal H-p$ an immediate calculation shows that $\iota_{\frac{\partial}{\partial p}|_{\Psi(x)}}g(x)(dH)_{\Psi(x)}=-g(x)$ and $\iota_{\frac{\partial}{\partial p}|_{\Psi(x)}}\iota_{(\Psi_*)_x(\gamma_x)}\omega_{\Psi(x)}=-(-1)^n \iota_{\frac{\partial}{\partial p}|_{\Psi(x)}}(dp)_{\Psi(x)}=-(-1) ^ n$, i.e. $g(x)=(-1)^n$, proving that the first assertion implies the second. On the other hand, the second condition immediately implies that $\Psi$ is a vortex $n$-plane for $h$, since for $w\in W=\{H=0\}$, $T_wW=\mathrm{ker}(dH)_w$.
\end{proof}

\begin{remark}\label{rem:lag}
	Assume that the Lagrangian $\mathcal L:J^1\pi\to \pi^*\Lambda^nT^ *\Sigma$ is (hyper-)regular with induced Hamiltonian section $h:P(\pi)\to M(\pi)$. If, furthermore, a volume form $\mathrm{vol}^\Sigma$ on $\Sigma$ is fixed, the problem of finding a section $\phi$ of $\pi:E\to \Sigma$ fulfilling the Euler-Lagrange equations is, by the preceding results, equivalent to finding a section $\tilde \Psi$ of $\tau:P(\pi)\to \Sigma$ such that $\forall x$
	\[
	\iota_{(\Psi_*)_x(\gamma_x)}\omega_{\Psi(x)}=(-1)^n (dH)_{\Psi(x)},
	\]
where $\Psi=h\circ \tilde \Psi$, $H$ is the Hamiltonian function associated to $h$ and $\mathrm{vol}^\Sigma$, and $\gamma\in \mathfrak X^n(\Sigma)=\Gamma(\Sigma,\Lambda^nT\Sigma)$ is uniquely determined by $\iota_\gamma(\mathrm{vol}^\Sigma)=1$. The Lagrangian variational problem is thus equivalent to the following ``Hamiltonian'' two-step problem:
\begin{itemize}
	\item Find $X_H\in \mathfrak X^n(M(\pi))$ such that $\iota_{X_H}\omega=(-1)^ndH$
	\item Find a section $\Psi$ of $M(\pi)\to \Sigma$ such that $(\Psi_*)_x(\gamma_x)=X_H(\Psi(x))$ for all $x$.
\end{itemize} 
(Note that $\Psi$ then automatically has values in $\{H=C\}\subset M(\pi)$ for an appropriate $C\in \mathbb R$, i.e. $\Psi$ factorizes through $h+C\cdot \tau^ *(\mathrm{vol}^\Sigma)$.) The equation $\iota_{X_H}\omega=(-1)^ndH$ can easily be generalized to the following ``{Hamilton-DeDonder-Weyl (or HDW) equations}''
\[\iota_{X_H}=-dH\]
for a couple $(H,X_H)\in \Omega^k(M(\pi))\times \mathfrak X^{n-k}(M(\pi))$. Typically a ``Hamiltonian $k$-form'' $H$ is given and the ``Hamiltonian $(n{-}k)$-vector field'' $X_H$ is considered to be the unknown. Solutions of the second equation $\Psi_*(\gamma)=X_H\circ \Psi$ for $H$ a 0-form are also called ``Hamiltonian $n$-curves'' cf. \cite{MR2105190}. These ideas generalize to the context of multisymplectic manifolds, which we will introduce in the following section.
\end{remark}

\section{Multisymplectic manifolds}
\label{sec:2}
Multisymplectic manifolds generalize the multiphase spaces crucial to the formulation of Hamiltonian classical field theories in Section \ref{sec:1}. Our definition is rather general but seems to be the most natural one, and is widely used by now in the mathematical literature. More restricted definitions amount to impose the existence of a global potential of the multisymplectic form and/or the existence of standard coordinate systems. The issue of normal forms will be treated by us in Section \ref{sec:3}. The main body of this section consists of examples showing that interesting classes of multisymplectic manifolds abound. We end this section by describing the Hamilton-DeDonder-Weyl equations on a general multisymplectic manifold.
 
\begin{definition}
	A ``multisymplectic'' manifold $(M,\omega)$ is a pair, where $M$ is a manifold, $k\geq 1$ and 
	$\omega\in \Omega^{k+1}_{cl}(M)$ is a closed differential form satisfying the following non-degeneracy condition: The map
	
\[\iota_\bullet\omega: TM\to \Lambda^{k}T^*M,~~v\mapsto \iota_v\omega\]
is injective. For fixed degree $k+1$ of the form such manifolds are also called ``$k$-plectic''. 
Such a form is sometimes simply called a ``multisymplectic form'' or a ``multisymplectic structure''.
\end{definition}

\begin{example}[The classical cases]$ $
	\begin{itemize}
		\item A symplectic manifold is, by definition, a 1-plectic manifold.
		\item An $n$-dimensional manifold equipped with a volume form is an $(n{-}1)$-plectic manifold.
	\end{itemize}
\end{example}

\begin{example}[Sums and products]
	As in the symplectic case, given two $k$-plectic manifolds $(M,\omega)$ and  $(\tilde M,\tilde \omega)$, there is a natural $k$-plectic structure $\pi_M^*\omega + \pi_{\tilde M}^*\tilde\omega$ on $M\times \tilde M$. Additionally $M\times \tilde M$ carries the multisymplectic structure given by $\pi_M^*\omega \wedge \pi_{\tilde M}^*\tilde\omega$, which is a multisymplectic manifold, even when $\omega$ and $\tilde \omega$ have different degrees.
\end{example}

\begin{example}[Multicotangent bundles and their subbundles]\label{exmulticotangent}
	Given a manifold $Y$ or a fibration $Y\overset{\pi}{\to} \Sigma$ the manifold $\Lambda^nT^ *Y$ resp. the manifolds $\Lambda^n_kT^*Y$ are multisymplectic. More generally we can consider an $(N{-}n)$-dimensional integrable distribution $V\subset TY$ instead of a fibration. Then $V$ would play the role of $ker(\pi_*)\subset TY$.
\end{example}

\begin{example}[Complex manifolds with holomorphic volumes] Let $(M,J)$ be a complex  manifold of dimension $m$, interpreted as a 2m-dimensional real manifold with an integrable almost-complex structure $J$. Let $\omega=\omega^{\mathbb R}+i\omega^{\mathbb I}\in \Omega^{m,0}(M)\subset \Omega^m(M)\otimes_{\mathbb R}\mathbb C$ be a holomorphic volume form, i.e. $\omega$ is a $\mathbb C$-valued smooth $m$-form, such that $\iota_{J(v)\omega}=i\iota_v\omega$ and $\partial \omega=\bar{\partial}\omega=0$. Then $\omega^{\mathbb R}$ and $\omega^{\mathbb I}$ are multisymplectic structures on $M$.
\end{example}

\begin{example}[Semisimple Lie groups]\label{exlie}
	Let $G$ be a real semi-simple Lie group. We construct a 2-plectic form on $G$ using the following facts:\label{lie}
	\begin{itemize}
		\item The Lie bracket is $Ad_g$-equivariant for all $g\in G$. As $G$ is semi-simple, we have $[\mathfrak g,\mathfrak g]=\mathfrak g$.
		\item The (symmetric) Killing-form $\langle\cdot,\cdot\rangle:\mathfrak g\times \mathfrak g\to \mathbb R$ is $Ad_g$-invariant for all $g\in G$ and $ad_X$ is a skew-adjoint linear map for all $X\in \mathfrak g$. It is non-degenerate for semi-simple Lie groups.
		\item The Maurer-Cartan 1-form $\theta^L\in \Omega^1(G,\mathfrak g)$ defined by $\theta^L_g=T_g( L_{g^{-1}}):T_gG\to T_eG=\mathfrak g$, where $L_g:G\to G$ is the left multiplication by $g$, is by construction left-invariant.
	\end{itemize}
	We define $\omega\in \Omega^3(G)$ by $\omega(u,v,w)=\langle\theta^L_g(u),[\theta_g^L(v),\theta^L_g(w)]\rangle$ $\forall g\in G, \forall u,v,w\in T_gG$. Non-degeneracy follows from $[\mathfrak g,\mathfrak g]=\mathfrak g$ and the non-degeneracy of the Killing form. The left-invariance of $\theta^L$ implies that $\omega$, too, is left-invariant. Using the description $Ad_g=T(L_g)\circ T(R_{g^{-1}})$, the $Ad_g$-invariance of the Killing form and the Ad-equivariance of the Lie bracket one can also show that $\omega$ is right-invariant. Any bi-invariant form on a Lie group is automatically closed, so $\omega$ is in $ \Omega_{cl}^3(G)$ and non-degenerate and thus defines a 2-plectic structure on $G$.
\end{example}

\begin{example}[$G_2$-structures]
	A closed $G_2$-structure for a seven-dimensional manifold $M$ is a closed differential 3-form $\omega$,such that for all $p\in M$, there exists a basis $e^1,...,e^7$ of $T^*_pM$ such that
	\[\omega_p=  e^{123}+e^{145}-e^{167}+e^{246}+e^{257}+e^{347}-e^{356},\] 
	where $e^{ijk}$ denotes $e^i\wedge e^j\wedge e^k$. Especially, for a closed $G_2$-structure $\omega$, the pair $(M,\omega)$ is a 2-plectic manifold. 
\end{example}

\begin{example}[Exact 2-plectic structure on $S^6$]\label{ex2pl00}
	We regard the standard closed $G_2$-structure on $\mathbb R^7$, given by 
	$$\omega=dx^{123}+dx^{145}-dx^{167}+dx^{246}+dx^{257}+dx^{347}-dx^{356}$$
	and pull it back to $S^6$ by the canonical inclusion $\rho:S^6\to \mathbb R^7$. This form $\rho^ *\omega$ is still closed, so for 2-plecticity we only need to verify its non-degeneracy. 
	
	Since the linear action of $G_2$ on $\mathbb R^7$ preserves $\omega$ and restricts to a transitive action on $S^6$ (in fact $Aut_{Lin}(\mathbb R^7,\omega)=G_2$ cf. eg. \cite{MR2253159}), it suffices to show non-degeneracy at one point. We regard the point $p=(0,0,0,0,0,0,1)\in S^6\subset \mathbb R^7$ and see 
	$$(\rho^*\omega)_p=\left(dx^{123}+dx^{145}+dx^{246}-dx^{356}\right)\big|_{T_pS^6}.$$
	This form is non-degenerate, as one can see, e.g., by applying Theorem \ref{thm:2pl6} from the next section or by direct verification. It follows that $(S^6,\rho^ *\omega)$ is a 2-plectic manifold, with a homogenous 2-plectic structure. As $H^3_{dR}(S^6)=0$, $\rho^ *\omega$ is exact.
\end{example}

\begin{remark}
	A more general construction for generating multisymplectic manifolds is described in \cite{MR2880222}. Their method recovers all homogenous strictly nearly K\"ahler 6-manifolds (especially $S^6$) as 2-plectic manifolds.
\end{remark}

\begin{example}[Exact 3-plectic structure on $S^6$]
	Let $R$ be the radial vector field $\sum x^i\frac{\partial}{\partial x^i}$ on $\mathbb R^7$. The differential 2-form $\tau=\rho ^ *(\iota_R\omega)$ with $\omega$ as in Example \ref{ex2pl00} is non-degenerate and $G_2$-invariant. However it is not symplectic, in fact $d\tau=3 (\rho ^*{\omega})$. (As one can see, upon using Section 4.1 of \cite{MR2253159}.) We have $\tau_p=\left(-dx^{16}+dx^{25}+dx^{34}\right)\big|_{T_pS^6}$, especially $\tau_p\wedge (\rho^*\omega)_p=0$. As both $\tau$ and $\rho ^ *\omega$ are $G_2$-invariant it follows
	\[
	d(\tau\wedge \tau)=2 d\tau\wedge\tau=6 (\rho^ *\omega)\wedge \tau=0.
	\]
	Thus $(S^6,\tau\wedge \tau)$ is a 3-plectic manifold, with a $G_2$-homogenous 3-plectic form. As $H^4_{dR}(S^6)=0$, this form is also exact.
\end{example}

As the above examples indicate, multisymplectic structures on closed manifolds do not, in general, give rise to cohomology classes. This is part of a very general phenomenon. In many degrees, multisymplectic structures exist in all cohomology classes (especially in the zero class).

\begin{theorem}[Genericity, Theorem 2.2 of \cite{MR0286119}] For $n\geq 7$ and $3\leq k\leq n-2$ an $n$-dimensional manifold has a $(k{-}1)$-plectic structure in every class in $H_{dR}^k(M)$. For such degrees the non-degenerate forms are $C^1$-open and dense in the (closed) forms.
\end{theorem}

Though in this article we do not study ``Hamiltonian dynamics'' on general multisymplectic manifolds in detail, we would like to give (and use) the following fundamental

\begin{definition}\label{defhddw}
	Let $(M,\omega)$ be an $n$-plectic manifold. A couple $(X,H)\in \Omega^k(M)\times \mathfrak X^{n{-}k}(M)$ with $0\leq k\leq n{-}1$ is called a solution of the ``Hamilton-DeDonder-Weyl (or HDW)'' equation if $$\iota_X\omega=-dH.$$
\end{definition}
\begin{remark}(Compare Remark \ref{rem:lag}.)\label{rem-end-2}
	\begin{enumerate}
		\item If either $H$ or $X$ is fixed in advance, the other element of a solution couple is also called a solution of the HDW equation $\iota_X\omega=-dH$.
		\item If $(X,H)$ is a solution of the HDW equation with $X$ an $m$-vector field (with $1\leq m\leq n$) and $\Sigma$ an $m$-dimensional manifold with a section $\gamma$ of $\Lambda^mT\Sigma$, then a smooth map $\Psi:\Sigma\to M$ is called a ``Hamiltonian $m$-curve'' (with respect to the Hamiltonian form $H$ on $M$) if $\forall x\in   \Sigma, (\Psi_*)_x(\gamma_x)=X_H(\Psi(x))$. Typically $\mathrm{vol}^\Sigma$ is a non-vanishing $m$-form on $\Sigma$ and $\gamma$ is the unique $m$-vector field satisfying $\iota_\gamma\mathrm{vol}^\Sigma=1$. Then there are local coordinates on $\Sigma$ such that $ \mathrm{vol}^\Sigma=dx^1\wedge ...\wedge dx^m$ and $\gamma=\frac{\partial}{\partial x^1}\wedge ....\wedge \frac{\partial}{\partial x^m}$, facilitating local computations.
	\end{enumerate}
\end{remark}

\section{Darboux type theorems}\label{sec:3}
A very important tool in symplectic geometry is the Darboux theorem stating that, given a point $p$ in a symplectic manifold $(M,\omega)$, there exist local coordinates $(x^1,...,x^{2m})$ near $p$ such that in these coordinates $\omega=dx^1\wedge dx^2+...+dx^{2m-1}\wedge dx^{2m}$. The existence of such coordinates relies on the two facts that all $2m$-dimensional symplectic vector spaces are linearly isomorphic and that locally a symplectic manifold $(M,\omega)$ is diffeomorphic to the linear symplectic manifold $(T_pM,\omega_p)$ for any $p\in M$. Neither of these results pertain to a general multisymplectic manifold $(M,\omega)$ and ``flatness'' of such a manifold, i.e. the existence of local coordinates such that $(M,\omega)$ can locally be identified with $(T_pM,\omega_p)$ for $p\in M$ turns out to be a rather special situation. We report in Subsection \ref{sub:31} on the linear multisymplectic case, without giving any proofs. After recalling the advantageous cases of symplectic and volume forms, and of certain ``multicotangent type manifolds'' in Subsections \ref{sub:symp} and \ref{sub:martin}, we give new results in Subsections \ref{sub:new1} and \ref{sub:new2}. In Subsection \ref{subs:2pl6} we recall the three possible cases for flat 2-plectic manifolds before giving an elementary construction to obtain 2-plectic forms on $\mathbb R^6$ that show that the ``linear type'' can change in a multisymplectic manifold, as well as that flatness may fail even when the linear type is constant throughout the manifold. Though similar examples exist in the literature, we included our constructions for their extreme simplicity.\\
In the last subsection, \ref{sub:lie}, we show that the canonical 2-plectic structure on a real simple Lie group is not flat unless the dimension of the Lie group is three. Though the linear type of these 2-plectic structures is constant, the result is rather natural but we were not aware of a proof of it in the literature.

\subsection{Linear types of multisymplectic manifolds}\label{sub:31}
In this subsection we will briefly discuss results concerning the linear types of multisymplectic manifolds. 
\begin{definition}
	A ``$k$-plectic vector space'' (over $\mathbb R$) is a pair $(V,\eta)$, where $V$ is a finite-dimensional $\mathbb R$-vector space and $\eta\in \Lambda^{k+1}V^*$ is non-degenerate, i.e. $\iota_\bullet\eta:V\to \Lambda^kV^*,~~v\mapsto \iota_v\omega$ is injective. A ``linear multisymplectomorphism'' $L$ between $(V,\eta)$ and $(\tilde V,\tilde \eta)$ is a linear isomorphism $L:V\to \tilde V$ satisfying $L^*\tilde \eta=\eta$. A ``$k$-plectic linear type''  is an isomorphism class of such pairs $(V,\eta)$. 
\end{definition}

\noindent Multisymplectomorphic vector spaces have equal dimensions, so we can ask: ``How many $(k{+}1)$-plectic linear types are there in dimension $n$?'' An answer is given by the following theorem:

\begin{theorem}Let $\Sigma^k_n$ denote the number of $(k{+}1)$-plectic linear types in dimension $n$. Then we have 
\begin{itemize}
	\item $\Sigma_n^n=1$ for all $n$, and $\Sigma^1_n$ as well as $\Sigma^{n{-}1}_n$ are zero for $n>1$.
	\item $\Sigma^2_n$ is 0 for $n$ odd and one for $n$ even.
	\item $\Sigma_{n}^{n-2}=\lfloor \frac{n}{2}\rfloor-1$, when $(n \mod 4)\neq 2$ (for $n\geq 4$) and
	$\Sigma_{n}^{n-2}= \frac{n}{2}$, when $(n \mod 4)=2$ (for $n\geq 4$).
	\item $\Sigma_6^3=3$, $\Sigma_7^3=8$  ,  $\Sigma_8^3=21$,  $\Sigma_7^4=15$ and $\Sigma_8^5=31$ .
 \item $\Sigma_n^k=\infty$ in all other cases.
\end{itemize}
\end{theorem}

\begin{proof}
	Most cases have been settled in \cite{MR0286119}. Three-forms in dimensions six, seven and eight have  been handled by \cite{MR0461545, MR630147, MR691457} and the remaining cases are settled in \cite{1609.02184}.
\end{proof}

\noindent For dimensions up to 10 the numbers look as follows, where the rows range from 0-forms (the ``$-$'' in the table) to $n$-forms:

\begin{align*}
\begin{array}{c|ccccccccccccccccccccc}
n=0&&&&&&&&&&&-\\[1.5em]
n=1&&&&&&&&&&-&&1\\[1.5em]
n=2&&&&&&&&&-&&0&&1\\[1.5em]
n=3&&&&&&&&-&&0&&0&&1\\[1.5em]
n=4&&&&&&&-&&0&&1&&0&&1\\[1.5em]
n=5&&&&&&-&&0&&0&&1&&0&&1\\[1.5em]
n=6&&&&&-&&0&&1&&3&& 3&&0&&1\\[1.5em]
n=7&&&&-&&0&&0&& 8&& 15&&2&&0&&1\\[1.5em]
n=8&&&-&&0&&1&& 21&&\infty&&31&&3&&0&&1\\[1.5em]
n=9&&-&&0&&0&&\infty&&\infty&&\infty&&\infty&&3&&0&&1\\[1.5em]
n=10&-&&0&&1&&\infty&&\infty&&\infty&&\infty&&\infty&& 5&&0&&1\\[1.5em]
\end{array}
\end{align*}
\subsection{Symplectic and volume forms}\label{sub:symp}
The next few subsections are motivated by these two classical theorems (\cite{MR0182927, MR0286137,MR997295}):
\begin{theorem}
	Let $(M,\omega)$ be a 1-plectic (i.e. symplectic) manifold of dimension $n=2m$ and $p\in M$. Then there exists a chart near $p$ $M\supset U\overset{\phi}\to \mathbb R^{2m}$ such that 
	\[\omega=\phi^*(dx^1\wedge dx^2+....+dx^{2m-1}\wedge dx^{2m}).\]
\end{theorem}

\begin{theorem}\label{voldarboux}
	Let $n\geq 2$ and $(M,\omega)$ be a $({n{-}1})$-plectic manifold of dimension $n$ (i.e. a manifold with a volume form), and $p\in M$. Then there exists a chart near $p$ $M\supset U\overset{\phi}\to \mathbb R^{n}$ such that 
	\[\omega=\phi^*(dx^1\wedge ....\wedge dx^{n}).\]
\end{theorem}

\noindent Each of these theorems can be decomposed into two statements:
\begin{enumerate}[(i)]
	\item Any symplectic form (resp. volume form) has the linear type of $dx^1\wedge dx^2+....+dx^{2m-1}\wedge dx^{2m}$ (resp. $dx^1\wedge ....\wedge dx^{n}$).
	\item Around any point $p$ the symplectic resp. $(n{-}1)$-plectic manifold $(M,\omega)$ is locally isomorphic to $(T_pM,\omega_p)$. (I.e. around $p$ there exists a chart $\phi:M\supset U\to T_pM$ such that $\phi(p)=0$ and $\phi^*\omega_p=\omega$.)
\end{enumerate}
As we have seen in the last subsection, there is no hope for (i) to hold for $k$ other than $1,n{-}1$. In the sequel we will investigate conditions for (ii) to hold. For this we will formulate the following property:

\begin{definition}
	A multisymplectic manifold $(M,\omega)$ is called ``flat near $p$'' for $p\in M$, if there exists a chart $\phi:U\to T_pM$ such that $\phi(p)=0$ and $\phi^*\omega_p=\omega$. It is called ``flat'' if it is flat near all $p$. Of course, $\omega_p$ is here interpreted as a constant-coefficient differential form on the manifold $T_pM$.
\end{definition}

\subsection{Multicotangent bundles}\label{sub:martin}
In this subsection we recall the situation for  multisymplectic manifolds, whose linear types correspond to that of a multicotangent bundle $(\Lambda^nT^*Y, \omega=-d\theta)$ from Example \ref{exmulticotangent}. 

\begin{definition}
	 A real $n$-plectic vector space $(V,\omega)$ is called ``standard'' if there exists a linear subspace $W\subset V$ such that $\forall u,v\in W$, $\iota_{u\wedge v}\omega=0$ and $$\omega^\#:W\to \Lambda^ n(V/W)^*, ~~~w\mapsto \left( 
	 (v_1+W,...,v_n+W)\mapsto \omega(w,v_1,...,v_n)
	 \right)$$
	 is an isomorphism.
\end{definition}

\begin{remark}
	In the above situation $W$ is unique if $n\geq 2$ and then often denoted $W_\omega$.
\end{remark}

\noindent From \cite{MR962194, MR1694063} the following result can easily be derived: 

\begin{theorem}\label{multithm} Let $n\geq 2$ and $(M,\omega)$ be a standard $n$-plectic manifold, i.e. $(M,\omega)$ has as constant linear type a fixed standard $n$-plectic vector space. Then $\mathcal W_\omega=\bigsqcup_{p\in M}W_{\omega_p}\subset \bigsqcup_{p\in M}T_pM=TM$ is a smooth distribution.
Furthermore, $(M,\omega)$ is flat if and only if $\mathcal W_\omega$ is integrable.
\end{theorem}

\subsection{Multisymplectic manifolds of product type}
In this subsection we study the local normal form for multisymplectic structures having as (constant) linear type the sum of $k$ $m$-dimensional vector spaces, each supplied with a volume form. It turns out that flatness arises exactly if all elements in a certain intrinsically defined collection of $m$-forms are closed.

\label{sub:new1}
\begin{theorem}\label{prodthm}
	Let $k\geq 2$, $m>2$ and $U\subset \mathbb R^{km}$ be open and $\omega\in\Omega^m_{cl}(U)$ be of linear type $dx^{1,2,...,m}+dx^{m+1,...,2m}+...+dx^{(k-1)m+1,...,km}$. Then there is a decomposition $\omega=\omega_1+...+\omega_k$, where $\omega_1,...,\omega_k\in \Omega^m(U)$ such that $rank(\omega_i)=m$. The forms $\omega_i$ are unique up to permutation.\\ 
	
	\noindent Furthermore, $(U,\omega)$ is flat if and only if $d\omega_i=0$ for all $i\in\{1,...,k\}$.
\end{theorem} 

\noindent The condition $rank(\omega_i)=m$ guarantees, that $(\omega_i)_p$ is decomposable for all $p$, i.e. a wedge product of one-forms. For the proof we need the following lemma:

\begin{lemma}\label{prodlemma}
	Let $V=\mathbb R^{km}$ where $k\geq 2$ and $m>2$ and $\{e^1,...,e^{km}\}$ dual to the standard basis $\{e_1,...,e_{km}\}$ of $\mathbb R^{km}$. Let $\alpha\in\Lambda^mV^*$ be given by $\omega=e^{1,2,...,m}+e^{m+1,...,2m}+...+e^{(k-1)m+1,...,km}$. Then, up to permutation, the forms $\omega_i=e^{(i-1)m+1,...,im}$ are the unique decomposable forms satisfying $\omega=\sum_{i=1}^n\omega_i$.
\end{lemma}

\begin{proof}
	Let $\{\tilde \omega_i\}$ be an alternative collection of decomposable forms with the above property, which are no permutation of $\{\omega_i\}$. We define the subspaces $\tilde E_i=\{v\in V|~ \iota_v\tilde\omega_j=0~\forall j\neq i\}$. Since we have by construction $V=\bigoplus \tilde E_i$, the projections $\tilde \pi_i:V\to \tilde E_i\subset V$ are well-defined. We can reconstruct $\{\tilde \omega_i\}$ from $\{\tilde E_i\}$ by setting $\tilde \omega_i=\tilde \pi_i^*(\omega|_{\tilde E_i})$. Hence, as $\{\tilde \omega_i\}$ is no permutation of $\{\omega_i\}$, $\{ \tilde E_i\}$ is no permutation of $\{ E_i\}$ (defined correspondingly). I.e, there exists a vector $v\in E_i$, which does not lie in a single $\tilde E_j$. As $v_i\in E_i$,  $\iota_v\omega=\iota_v\omega_i$ is decomposable. However, $\iota_v\omega=\iota_v(\sum\tilde \omega_i)$ has several nonzero summands, i.e. is not decomposable, which yields a contradiction. Hence any collection  $\{\tilde \omega_i\}$ of decomposable forms summing up to $\omega$ is a permutation of $\{\omega_i\}$.
\end{proof}

\begin{proof}[Proof of the Theorem]
	By the preceding lemma we know that forms $\omega_i$ exist pointwise. To prove their smoothness, we begin with showing that the distributions $E_i$, defined in Lemma \ref{prodlemma}, are smooth, i.e. subbundles. Assume $U$ to be open and constractible. Then there is a canonical isomorphism $TU=U\times \mathbb R^{km}$. We consider $\omega$ as a map $U\to \Lambda^m\mathbb (R^{km})^*$. As $\omega$ is of constant linear type, it maps into \[\eta\cdot  GL(\mathbb R^{km})=(e^{1,2,...,m}+e^{m+1,...,2m}+...+e^{(k-1)m+1,...,km})\cdot  GL(\mathbb R^{km})\subset\Lambda^m\mathbb (R^{km})^* .\]
	By the above lemma, the stabilizer of $\eta$ is isomorphic to $S_k\ltimes SL(\mathbb R^d)^k$, where $S_k$ is the permutation group of $k$ elements. We regard the following diagram:
	\[
	\xymatrix{
		&&\frac{GL(\mathbb R^{km})}{ SL(\mathbb R^m)^k}\ar[d]^{\pi_{\sigma}}\ar[r]^{\pi~~~~~ }&\frac{GL(\mathbb R^{km})}{ GL(\mathbb R^m)\times GL(\mathbb R^{(k-1)m})}\ar[r]^\cong&Gr_{m}(\mathbb R^{km})\\
		U \ar@{.>}[rru]^{s_i}   \ar[r]_{\omega~~~~~~}&\eta\cdot  GL(\mathbb R^{km})\ar[r]_{\cong}& \frac{GL(\mathbb R^{km})}{S_k\ltimes SL(\mathbb R^m)^k}&&},
	\]
	where $\pi$ is induced by the inclusion of $SL(\mathbb R^m)^{k-1}$ into $GL(\mathbb R^{(k-1)m})$ and $Gr_m(\mathbb R^{km})$ is the Grassmann manifold of all $m$-dimensional vector subspaces of $\mathbb R^{km}$.
	The map $\pi_\sigma$ is a $k!$-fold covering and $U$ is contractible, so the horizontal map admits $k!$ sections. We choose one section for each orbit of $S_{k-1}$, the stabilizer of $\{1\}$ of the $S_k$-action on $\{1,...,k\}$, acting on $\frac{GL(\mathbb R^{km})}{ SL(\mathbb R^m)^k}$ and denote them as $s_1,...,s_k$. Composed with $\pi$, we get $k$ smooth maps $\pi\circ s_i:U\to Gr_d(\mathbb R^{km})$. By the definition of $Gr_m(\mathbb R^{km})$ they yield $k$ smooth subbundles $E_i$ of $TU$, which correspond pointwise to the $E_i$ of the above lemma. Thus the elements $\omega_i=\omega|_{E_i}$ are smooth.\\
	
	Obviously, if $\omega$ is flat, then the $\omega_i$ are closed. Conversely assume all $\omega_i$ are closed. Then the $(k-1)m-forms$
	$\Omega_i=\omega_1\wedge \omega_2 \wedge ...\widehat{\omega_i}\wedge ...\wedge \omega_k$ are also closed. Consequently the subbundles $E_i=\{v|\iota_v\Omega_i=0\}$ are involutive and hence integrable.  Also, for any $I\subset \{1,...,k\}$ the sums $\bigoplus_{i\in I} E_i$ are integrable by the same argument. Especially $E_{i}'=\bigoplus_{j\neq i} E_j$ is integrable. Thus for any $p\in U$ there exist open sets $U_i\subset M$ containing $p$ and submersions $\phi_i:U_i\to \phi_i(U_i)\subset_{\text{open}}\mathbb R^m$, satisfying $\ker(D\phi_i)=E_{i}'|_{U_i}$. Then automatically $D\phi_i|_{E_i}:E_i|_{U_i}\to T\mathbb R^m$ is injective and thus there exists an open neighbourhood $V\subset \bigcap U_i$ of $p$ on which $\Phi=(\phi_1,...,\phi_k):V\to (\mathbb R^m)^k$ is a diffeomorphism onto its image, i.e. a chart. We know that the pullbacks $(\Phi^{-1})^*\omega_i$ are closed and of the form 
	\[f_idx^{(i-1)m+1}\wedge ...\wedge dx^{im},\]
	so $f_i$ only depends on $x^{(i-1)m+1},...,x^{im}$. The theorem then follows from applying the Darboux theorem for volume forms to the $(\Phi^{-1})^*\omega_i$.
	(For a similar statement proven differently cf. also \cite{MR3146582}.)
\end{proof}

\subsection{$(m{-}1)$-plectic complex $m$-manifolds}\label{sub:new2}

We consider here, for $m>2$, (2m)-dimensional real manifolds with a $(m{-}1)-plectic$ structure having as (constant) linear type the real part of a complex volume form, and show that such multisymplectic manifolds are flat if and only if a certain associated almost-complex structure is integrable.

\begin{theorem}\label{clxthm}
	Let $m> 2$ and $U\subset \mathbb R^{2m}$ be open and $\omega\in\Omega^m_{cl}(U)$ be of linear type $\mathbb Re((dx^1+idx^2)\wedge ...\wedge(dx^{2m-1}+idx^{2m}) )$, where $m>2$. Then, up to sign, there is a unique almost-complex structure $J$ such that the following equality holds for all $p\in U$ and $v,w\in T_pU$: 
	\begin{align}\label{clx-eq}
		\iota_{J(w)}\iota_v\omega=\iota_{w}\iota_{J(v)}\omega
	\end{align}
	
	\noindent Furthermore, $(U,\omega)$ is flat if and only if $J$ is integrable.
\end{theorem} 

 \noindent For the proof we need the following lemma from \cite{MR2462806}:

 \begin{lemma}
 Let $m>2$ and $J$ a linear complex structure on the $2m$-dimensional real vector space $V$. Let $\omega=\omega^\mathbb R+i\omega^{\mathbb I}\in \Lambda^{m,0}V^*$ be non-zero. Then 
 \[
 \mathcal A_{\omega^ \mathbb R}=\{A\in End_{\mathbb R}(V)| \iota_{A(w)}\iota_v\omega^\mathbb R=\iota_{w}\iota_{A(v)}\omega^\mathbb R\}=\mathbb R \cdot \text{id} \oplus \mathbb R\cdot J
 \]
 \end{lemma}

\begin{proof}
	The ``$\supset$''-inclusion is clear. For the other inclusion we first observe, that the elements of $\mathcal A_{\omega^ \mathbb R}$ commute:
	\[
	\omega^ \mathbb R(ABv,w,x, \ldots)=\omega^ \mathbb R(v,Aw,Bx, \ldots)=\omega^ \mathbb R(BAv,w,x, \ldots).
	\]
	Especially $\mathcal A_{\omega^ \mathbb R}\subset End_{\mathbb C}(V)$, as every element has to commute with $J$. Moreover,  any element $A\in \mathcal A_{\omega^ \mathbb R}$ has to be diagonal as a complex matrix. To see that, we observe that $A(v)$ is always $\mathbb C$-linearly dependent on $v$. We have
	\[
	\omega^ \mathbb R(v,Av,x, \ldots)=\omega^ \mathbb R(v,v, Ax, \ldots)=0,
	\]
	so $\iota_{v}\iota_{A(v)}\omega^ \mathbb R=0$ for all $v$. As $\omega^ \mathbb R$ is at least a 3-form, this implies $\iota_{v}\iota_{A(v)}\omega=0$ for all $v$. Now $\omega$ is a complex volume, so $v$ is a complex eigenvector of $A$. In remains to show, that all eigenvalues are equal, but this follows from
	\[
	\lambda_1 \cdot 	\omega^ \mathbb R(v_1,v_2, \ldots)=\omega^ \mathbb R(Av_1,v_2, \ldots)=\omega^ \mathbb R(v_1,Av_2, \ldots)=	\lambda_2 \cdot 	\omega^ \mathbb R(v_1,v_2, \ldots),
	\] 
	again using the fact, that $\omega^ \mathbb R$ is the real part of a complex volume form.
\end{proof}

\begin{proof}[Proof of Theorem \ref{clxthm}]
	 At any point we choose $J_p$ to be the unique almost-complex structure compatible with the standard orientation on $U$ and satisfying \eqref{clx-eq}, existing by the above lemma. The smoothness of $\omega$ assures that the almost-complex structure $J$ varies smoothly. If $(U,\omega)$ is flat, then with respect to some chart $J$ has constant coefficients, i.e. is integrable.
	On the other hand if $J$ is integrable, then we can extend $\omega$ to an element $\omega^{\mathbb C}$ of $\Omega^{m,0}(U)$, by $\iota_v\omega^{\mathbb C}=\iota_v\omega - i \cdot \iota_{J(v)}\omega$. By the integrability of $J$, this form is still closed, i.e. a holomorphic volume form (of the complex manifold $(U,J)$). By the holomorphic version of the Darboux-Moser Theorem for volume forms, there exist holomorphic local coordinates $(z_1,...,z_n)$, such that $\omega^{\mathbb C}$ (and hence $\omega=\mathbb Re(\omega^{\mathbb C})$) has constant coefficients.
\end{proof}

\begin{remark}
	For odd $m$, the almost-complex structures defined by Equation \eqref{clx-eq}, could also be described by the equation
	\[
	(\iota_v\omega^\mathbb R) \wedge \omega^\mathbb R=\pm \iota_{J(v)}\left(\frac{1}{2}\omega^\mathbb R\wedge \omega^{\mathbb I}\right),
	\]
	as has been done for the $m=3$ case in, for example, \cite{MR2253159,MR1863733}.
\end{remark}

\subsection{2-plectic 6-manifolds}\label{subs:2pl6}
We construct here new 2-plectic structures on $\mathbb R^6$ that do not have constant linear type respectively are not flat despite having constant linear type, showing that flatness of multisymplectic manifolds is a subtle issue.

For non-degenerate three-forms in dimension six there are three linearly inequivalent normal forms. We will recall their flatness conditions, as described in \cite{126197}. The different cases were  presented in \cite{ MR2069397,MR2437342,MR1859316,MR962194,MR0238225}.

\begin{theorem}\label{thm:2pl6}
	Let $U\subset \mathbb R^6$ be open and $\omega\in \Omega^3_{cl}(U)$ (possibly degenerate). Choose any volume form $\Omega\in \Omega^6(U)$. There is a unique $J\in \Gamma(U,End(TU))=C^\infty(U,\mathbb R^{6\times 6})$ satisfying $(\iota_v\omega)\wedge \omega=\iota_{J(v)}\Omega$ for all $v\in TU$.
	Then we have:
	\begin{enumerate}[(i)]
		\item If $trace(J(p)^2)>0$, then $\omega_p$ has the linear type of $e^1\wedge e^2 \wedge e^3+e^4\wedge e^5\wedge e^6$. If this is the case on an open subset $V\subset U$, then $\omega|_V=\omega_1+\omega_2$ for decomposable forms $\omega_1,\omega_2\in \Omega^3(V)$, unique up to order. In such cases $(V,\omega|_V)$ is flat, if and only if $d\omega_1=d\omega_2=0$. 
		\item If $trace(J(p)^2)<0$, then $\omega_p$ has the linear type of $e^1\wedge e^3 \wedge e^5-e^1\wedge e^4 \wedge e^6- e^2\wedge e^3 \wedge e^6-e^2\wedge e^4 \wedge e^5$. If this is the case on an open subset $V\subset U$, then $\tilde J=\sqrt{\frac{-6}{trace(J^2)}}\cdot J$ defines an almost-complex structure.\\
		In such cases $(V,\omega|_V)$ is flat, if and only if $\tilde J$ is an integrable almost-complex structure. 
		\item If $trace(J(p)^2)=0$ and $\omega_p$ is non-degenerate, then $\omega_p$ has the linear type of $e^1\wedge e^5\wedge e^6 -e^2\wedge e^4\wedge e^6+e^3\wedge e^4\wedge e^5$. If this is the case on an open subset $V\subset U$, then $E=ker(J)\subset TV$ yields a distribution.\\
		In such cases $(V,\omega|_V)$ is flat, if and only if $E$ is an integrable distribution. 
	\end{enumerate}
\end{theorem}
\begin{proof}
	The linear statements are proven in \cite{MR2253159}. The three cases case can be reduced to special cases of Theorems \ref{prodthm}, \ref{clxthm} and \ref{multithm}.
\end{proof}

We will use the above theorems to construct 2-plectic 6-manifolds not having constant linear type or flatness properties. Similar constructions have been investigated in \cite{MR2069397,MR2437342} and other examples arise in the theory of special holonomy cf. (\cite{MR1939543,MR916718}). We will construct our examples using the following lemma (compare also the preprint \cite{1608.07424}).

\begin{lemma}\label{lem:2pl}
	Let $M=\mathbb R^6$ and 
	\[
	\omega=\omega^f= dx^{135} -dx^{146}-dx^{236} + f(x)\cdot dx^{245}\in \Omega^3(M),\]
	where $f:\mathbb R^6\to \mathbb R$ only depends on $x^2,x^4$ and $x^5$. Then $(M,\omega)$ is a multisymplectic manifold. Furthermore, $\omega_x$ is of linear type $(i)$ when $f(x)>0$, $(ii)$ when $f(x)<0$ and $(iii)$ when $f(x)=0$, using the numbering from the above theorem.
\end{lemma}
\begin{proof}
	The proof is a simple consequence of the above theorem, by explicit calculation of $J$ and noticing that $dx^{135} -dx^{146}-dx^{236}$ is non-degenerate. With respect to the standard volume on $\mathbb R^6=T_x\mathbb R^6$, we obtain that 
	\[
	J(x)=\begin{pmatrix}
	0&-2f(x) &0&0&0&0\\
	-2&0&0&0&0&0\\
	0&0&0&-2f(x) &0&0\\
	0&0&-2&0&0&0\\
	0&0&0&0&0&2\\
	0&0&0&0&2f(x)&0\\
	\end{pmatrix}.
	\]
	Squaring and taking the trace completes the proof.
\end{proof}

\begin{example}[Non-constant linear type]\label{ex-r6-linear-vary}
	We set $f=x^2$. Then in any neighbourhood of $0\in M$, there exist points where $f$ is positive and points, where $f$ is negative. Hence $(M,\omega^f)$ does not have constant linear type around 0. Consequently, it can not satisfy the Darboux property at 0. 
\end{example}

\begin{remark}
	An example of non-constant linear type for non-degenerate four-forms in dimension six can be given as follows. Let $M= \mathbb R^6$ and $\omega=dx^{1234}+dx^{1256} +x^3 dx^{3456}$. At $x^3=0$ the type changes.
\end{remark}

\begin{example}[Constant linear type but non-flat]\label{nflat}
	We regard the multisymplectic submanifold $M^{>0}=\{x\in\mathbb R^6|~x^2>0\} \subset M$ from Example \ref{ex-r6-linear-vary}. We set
	
	\begin{minipage}{0.4\textwidth}
		\begin{align*}
		&\alpha_1=(\sqrt {x^2} dx^2-dx^1)\\
		&\alpha_2=(\sqrt {x^2} dx^4-dx^3)\\
		&\alpha_3=(\sqrt {x^2} dx^5+dx^6)\\
		\end{align*}
	\end{minipage}
	\begin{minipage}{0.35\textwidth}
		\begin{align*}
		&\alpha_4=(\sqrt {x^2} dx^2+dx^1)\\
		&\alpha_5=(\sqrt {x^2} dx^4+dx^3)\\
		&\alpha_6=(\sqrt {x^2} dx^5-dx^6)\\
		\end{align*}
	\end{minipage}
	
	and
	\[
	\omega_1= \frac{1}{(2\sqrt{x^2})} \alpha_1\wedge \alpha_2\wedge \alpha_3\]\[
	\omega_2=\frac{1}{(2\sqrt{x^2})}\alpha_4\wedge \alpha_5\wedge \alpha_6.
	\]
	
	It follows that
	\[
	\omega= \omega_1+\omega_2
	\text{~~~~~~~~and~~~~~~~~}
	(\omega_1\wedge \omega_2)_p\neq 0~~\forall p\in M^{>0},\]
	and the linear type is thus constantly type (i) from Theorem \ref{thm:2pl6}. We observe that $\omega_1=\frac{1}{2}\omega +\frac{1}{2}\sqrt{x^2}(dx^{246}-dx^{235}-dx^{145}) +\frac{1}{2\sqrt{x^2}}dx^{136}$ and hence 
	\[
	d\omega_1=\frac{1}{4\sqrt{x^2}}dx^{1245} +\frac{1}{4\sqrt{x^2}^3}dx^{1236}.
	\]
	
\noindent As $d\omega_1\neq 0$ on any nonempty open subset of $M^{>0}$, we know that $(M^{>0},\omega)$ is nowhere flat.
\end{example}

\begin{remark}
	Similar examples can be constructed already for three-forms in $\mathbb R^5$. In \cite{MR801210} it is shown that $(dx^{12}+dx^{34})\wedge (dx^5+x^2dx^4)\in \Omega^3(\mathbb R^5)$ is nowhere flat, although it is non-degenerate, closed and has constant linear type.
\end{remark}

\subsection{2-plectic Lie groups}
\label{sub:lie}
We prove that the canonical 2-plectic structure on a simple Lie group is flat only if the group is three-dimensional.

\begin{theorem}\label{liegrp}
	Let $(G,\omega)$ be a real simple Lie group with its canonical three-form, as described in Example \ref{lie}. Then $(G,\omega)$ has constant linear type but is flat if and only if its dimension is three.
\end{theorem}
\begin{proof}
Constancy of linear type follows immediately from the bi-invariance of $\omega$. Without loss of generality, we can assume for the rest of the proof, that $G$ is connected and simple. In the three-dimensional case the flatness is a consequence of the Darboux theorem for volume forms (Theorem \ref{voldarboux}). For all real simple Lie groups of dimension higher than three, we have $$Aut(\mathfrak g,\omega_e )=Aut(\mathfrak g,[\cdot,\cdot])\subset Aut(\mathfrak g,\langle\cdot,\cdot\rangle ),$$
where the leftmost and rightmost terms are linear automorphisms preserving the respective tensor and the middle term are the Lie algebra automorphisms of $\mathfrak g$. The left equality is the statement of Theorem 2.2 of  \cite{MR3054295} and the right inclusion follows, because the Killing form is intrinsically defined from the Lie bracket.\\
Let us assume that $G$ admits a chart $\phi:U\subset G\to  V\subset\mathfrak g$ near $e$, such that $(T_g\phi)^*\omega_e=\omega_g$, where $\omega_e$ should be interpreted as the constant coefficient extension of $\omega_e\in \mathfrak g=T_e\mathfrak g$. The natural left-invariant pseudo-Riemannian metric on $G$ is defined by $h_g=-(\theta^L_g)^*\langle \cdot ,\cdot\rangle$, where $\theta_g^L:T_gG\to \mathfrak g$ is the Maurer-Cartan one-form. By construction we have
$$(\theta^L_g)\circ(T_g\phi)^{-1}\in Aut(\mathfrak g,\omega_e).$$
 So $(\theta^L_g)\circ(T_g\phi)^{-1}$ preserves $h_e=-\langle\cdot,\cdot\rangle $, i.e.
 \[
 (T_g\phi)^*h_e=(T_g\phi)^*((\theta^L_g)\circ(T_g\phi)^{-1})^*h_e=(\theta^L_g)^*h_e=h_g
 \]
 This means that $\phi$ is a flat chart for $(G,h)$, where $h$ is the canonical left-invariant metric on $G$. Such a chart can not exist, because real simple Lie groups with canonical left-invariant metric have non-zero curvature (cf. e.g. \cite{MR719023}).
\end{proof}

\section{The group of multisymplectic diffeomorphisms}

In the last section we studied the local structure of multisymplectic manifolds. In this section we will investigate the diffeomorphisms preserving this structure. 

\begin{definition} A ``{local diffeomorphism}'' $\varphi$ of $M$ is a diffeomorphism between two open subsets $U,V$ of M. It is called ``local multisymplectic diffeomorphism'' if it satisfies $\varphi^*(\omega|_V)=\omega|_U$. The pseudogroup of local multisymplectic diffeomorphisms is called $\text{Diff}_{loc}(M,\omega)$. Its subgroup of global diffeomorphisms is denoted by $\text{Diff}(M,\omega)$ and called ``group of multisymplectic diffeomorphisms or multisymplectomorphisms'' of $(M,\omega)$. The elements of the Lie algebra $\mathfrak{X}(M,\omega)=\{X|\mathcal L_X\omega=0\}\subset \mathfrak X(M)$ are called ``multisymplectic or locally Hamiltonian vector fields''.
\end{definition}

\noindent We will consider the following question:

\begin{quote}
	``Let $(M,\omega)$ be multisymplectic. How transitive is the action of $\text{Diff}(M,\omega)$ on $M$?''
\end{quote}

\noindent We will distinguish several degrees of transitivity:
\begin{definition}
	Let $X$ be a set and $G\times X\to X, (g,p)\mapsto g(p)$ a group action. The action is called ``$k$-transitive'', if for any two k-tuples $(p_1,...,p_k)$, $(q_1,...,q_k)$  of elements in $X$ satisfying $p_i\neq p_j$ and $q_i\neq q_j$ for $i\neq j$ there exists an element $g\in G$ such that $g(p_i)=q_i$ for $i=1,...,k$.
\end{definition}

In this section we will answer this question for several classes of examples. First, in Subsection \ref{sub:41} we will review the classical cases of symplectic and volume-preserving diffeomorphisms and show that the multisymplectomorphism group of $\mathbb C^n$ with the real part of a complex volume form acts $k$-transitively for all $k$. Subsection \ref{sub:42} will treat several situations, where the multisymplectic diffeomorphisms act 1-transitively but not 2-transitively, including some examples from the last section and those discussed in \cite{MR962194}. Finally, we will briefly discuss examples, where the action is not even 1-transitive.

\subsection{Very transitive cases}\label{sub:41}
The following theorem shows that the multisymplectomorphisms of symplectic and volume forms are very transitive:

\begin{theorem}[\cite{MR0236961}]\label{moser}
	Let $(M,\omega)$ be a connected symplectic manifold or a connected manifold equipped with a volume form. Then $\text{Diff}(M,\omega)$ acts k-transitively on $M$ for all $k\in\mathbb N$.
\end{theorem}

\begin{corollary}\label{sympcor}
	Let $(M,\omega)$ be a connected symplectic $2n$-dimensional manifold. Then $(M,\omega^j)$ is a multisymplectic manifold for all $j\in\{1,...,n\}$. Moreover $\text{Diff}(M,\omega^j)$ always acts $k$-transitively for all $k$. 
\end{corollary}

\begin{proof}
	The non-degeneracy of $\omega^j$ follows from $\iota_X(\omega^j)\wedge \omega^{n-j}=j\cdot (\iota_X\omega)\wedge \omega^{n{-}1}=\frac{j}{n}\iota_X(\omega^n)$. For $X\neq 0$ the latter is non-zero, as $\omega^n$ is a volume form.  The second statement follows immediately as $\text{Diff}(M,\omega)\subset \text{Diff}(M,\omega^j)$.
\end{proof}

\noindent We note that infinitesimally the converse is also true, except for the case $j=n$:
\begin{lemma}
	Let $(M,\omega)$ be as in Corollary \ref{sympcor} with $n>1$. Then $\mathfrak X(M,\omega^j)=\mathfrak X(M,\omega)$ for $1\leq j<n$ and $\mathfrak X(M,\omega^n)\supsetneq \mathfrak X(M,\omega)$.
\end{lemma}

\begin{proof}
	The $j=n>1$ case is a consequence of Gromovs non-squeezing theorem, cf. e.g. \cite{MR809718}.	If $j<n$ we calculate: \[\mathcal{L}_X(\omega^j)=j\cdot \mathcal L_X\omega \wedge \omega ^{j-1}.\]
	So  $\mathcal L_X\omega=0$ implies $\mathcal{L}_X(\omega^j)=0$. But on the other hand $\cdot \wedge \omega^j: \Omega^2(M)\to \Omega^{2+2j}(M)$ is injective for $2j\leq 2n-2$ by Lepage's divisibility Theorem, which is stated below. 
\end{proof}

\begin{theorem}[Lepage's divisibility Theorem, see, e.g., \cite{MR882548}]\label{lepage}
	Let $\Omega\in \Lambda^2(\mathbb R^{2n})^*$ be non-degenerate. Then the following map is a bijection for $0\leq p<n$:
	\[\Lambda^p(\mathbb R^{2n})^*\to \Lambda^{2n-p}(\mathbb R^{2n})^*,~\eta\mapsto \eta\wedge \Omega^{n-p}.\] 
\end{theorem}

Another class of multisymplectic manifolds with very transitive multisymplectomorphism groups arises from complex volume forms.

\begin{theorem}\label{transC}
	Let $n\geq 2$, $M=\mathbb C^n\cong \mathbb{R}^{2n}$ and $\omega=Re(dz^1\wedge ...\wedge dz^n)$. Then $\text{Diff}(M,\omega)$ acts $k$-transitively on $M$ for all $k$. 
\end{theorem}
We will prove the stronger statement, that the group $Aut_1^{alg}(\mathbb C^n)\subset \text{Diff}(M,\omega)$ of polynomial biholomorphisms of determinant one acts $k$-transitively on $\mathbb C^n$ for all $k$. The idea of the proof below was explained to us by Frank Kutzschebauch who showed a much more general result on biholomorphism groups in \cite{MR3667419}. 
\begin{proof}
Let $p_j=(x_j,y_j,z_j)_{1\leq j\leq k}\in \mathbb C\times \mathbb C^{n-2}\times \mathbb C$ be pairwise different points.
\begin{enumerate}
		\item Without loss of generality, we may assume that all $x_i$ are pairwise different. To see this, we will find a map $T$ in $SL(n,\mathbb C)\subset Aut_1^{alg}(\mathbb C^n)$ such that $T(p_i)$ have different first components for $i\in \{1,...,k\}$. Consider the $k\choose 2$ hyperplanes $H_{ij}=\{\psi\in (\mathbb C^n)^*|\psi(p_i-p_j)=0\}\subset (\mathbb C^n)^*$ for $1\leq i\leq j\leq k$. As there are only finitely many $H_{ij}$, the space $(\mathbb C^n)^*\backslash \bigcup H_{ij}$ is non-empty. Let $\phi_1\in(\mathbb C^n)^*\backslash \bigcup H_{ij}$. Extend $\phi_1$ to a basis $\{\phi_1,...,\phi_n\}$ of $(\mathbb C^*)^n$. Then $\tilde T(p):=(\phi_1(p),...,\phi_n(p))$ is a linear isomorphism, such that $(T(p_j))_{j\in\{1,...,k\}}$ have different first components. We get the desired map by setting $T=\lambda \tilde T$ for an appropriate $\lambda\in \mathbb C\backslash\{0\}$.
		\item There is an algebraic automorphism with Jacobian determinant 1, moving $(x_j,y_j,z_j)$ (with $x_i$ pairwise different) to $(x_j,0,j)$. Let $P:\mathbb C\to \mathbb C^{n-2}$ be a polynomial satisfying $P(x_j)=y_j$ for $j\in\{1,...,k\}$ and $Q:\mathbb C\to \mathbb C$ a polynomial satisfying $Q(x_j)= z_j-j$ for $j\in\{1,...,k\}$. Then the desired algebraic automorphism is given by $(x,y,z)\mapsto (x,y-P(x),z-Q(x))$. Note that no polynomial $P$ is necessary when $n=2$.
		\item There is an algebraic automorphism with Jacobian  determinant 1, moving $(x_j,0,j)$ (with $x_i$ pairwise different) to $(0,0,j)$. Let $P:\mathbb C\to \mathbb C$ be a polynomial satisfying $P(j)=x_j$ for $j\in\{1,...,k\}$. Then the automorphism $(x,y,z)\mapsto (x-P(z),y,z)$ has the desired property.
		\item By composing steps 1., 2. and 3. we can construct an algebraic biholomorphism $\Psi$ of determinant 1, such that $\phi(p_j)=(0,0,j)$ for $j\in\{1,...,k\}$. Given an alternative collection of points $\tilde p_1,...,\tilde p_k$ we can construct $\tilde \Psi$ in the same manner. Then $\tilde \Psi\circ \Psi^{-1}$ is an algebraic automorphism of Jacobian determinant 1 such that $\tilde \Psi\circ \Psi^{-1}(p_j)=\tilde p_j$ for $j\in \{1,...,k\}$. 
\end{enumerate}
\end{proof}

\begin{example} The multisymplectomorphisms of the manifold $M=\mathbb R^6$, $\omega=dx^{135}-dx^{146}- dx^{236}-dx^{245}$ act $k$-transitively on $M$ for all $k$. This example is just the real description of the $(n=3)$-case of the above theorem.
\end{example}

\subsection{Slightly transitive cases}\label{sub:42}

In this subsection we will treat several examples, where the action of $\text{Diff}(M,\omega)$ is 1-transitive but not 2-transitive. To identify those cases, we will use the following criterion.

\begin{lemma}\label{lemfol}
	Let $M$ be an $n$-dimensional manifold. If a group $G\subset \text{Diff}(M)$ preserves a  regular foliation $\mathcal F$ of dimension $r\not\in\{0,n\}$, its action on $M$ is not 2-transitive.
\end{lemma}

\begin{proof}
	Let $\phi$ be a diffeomorphism preserving $\mathcal F$. If $p_1,p_2$ are in the same leaf $F$, then $\phi(p_1), \phi(p_2)$ have to be in the same leaf $\phi(F)$. As $r$ is required to be different from $n$ several leaves exist, and as $r$ is nonzero each leaf contains many points. We take leaves $F_1\neq F_2$ and $p_1\neq p_2$ in $F_1$ and $q_1\in F_1$ and $q_2\in F_2$, then there is no $\phi\in G$, such that $\phi(p_j)=q_j$ for $j=1,2$.
\end{proof}

Using this criterion, we will first analyse a few flat examples from the last subsection, and then give an analysis of the non-flat case built in Example \ref{nflat}.

\begin{example}\label{ediff1}
	Let $(M,\omega)=(\mathbb R^6, dx^{156} -dx^{246}+dx^{345})$, i.e. the flat model with the third linear type from  Subsection \ref{subs:2pl6}. By Theorem \ref{thm:2pl6} the integrable distribution $E$ generated by $\{\frac{\partial}{\partial x^1},\frac{\partial}{\partial x^2},\frac{\partial}{\partial x^3}\}$ is preserved by $\text{Diff}(M,\omega)$ as it is constructed naturally only using $\omega$. Hence $\text{Diff}(M,\omega)$ does not act 2-transitively on $M$. However all translations are multisymplectomorphisms. Thus, in this case, $\text{Diff}(M,\omega)$ acts transitively, but not 2-transitively on $M$.
\end{example}

\begin{remark}\label{rdiff1}
	This example can be extended to all multisymplectic manifold built as in Example \ref{exmulticotangent}. By Theorem \ref{multithm} the multisymplectomorphisms of these manifolds preserve the (foliation given by the) fibers of $\pi$, hence they do not act 2-transitively on $M$. In \cite{MR962194, MR2105190} the multisymplectomorphism groups are explicitly calculated. They are isomorphic to $\text{Diff}(Q)\ltimes \Omega_{cl}^n(Q)$, i.e. they consist of diffeomorphisms of the base and translations by closed forms on the fibres.
\end{remark}

\begin{example}\label{ediff2}
	Let $(M,\omega)=(\mathbb R^6, dx^{123}+dx^{456})$. By Theorem \ref{thm:2pl6} the forms $\omega_1= dx^{123}$ and $\omega_2=dx^{456}$ are either preserved or interchanged by multisymplectic diffeomorphisms. By an argument analogous to Lemma \ref{lemfol}, we see that starting with two points which are in different leaves with respect to both the foliations generated by $\{\frac{\partial}{\partial x^1},\frac{\partial}{\partial x^2},\frac{\partial}{\partial x^3}\}$ respectively $\{\frac{\partial}{\partial x^4},\frac{\partial}{\partial x^5},\frac{\partial}{\partial x^6}\}$, we can not arrive at a pair of points which share the same $(x^1,x^2,x^3)$-coordinates.
	Again, we can achieve 1-transitivity by translations. In conclusion, $\text{Diff}(M,\omega)$ acts transitively, but not 2-transitively on $M$. 
\end{example}

\begin{remark}
	This example also can be generalized to the setting of Theorem \ref{prodthm}. For $m>2$ and $k>1$ the multisymplectic manifold $M=\mathbb R^{km}$, $\omega=dx^{1,2,...,m}+dx^{m+1,...,2m}+...+dx^{(k-1)m+1,...,km}$ satisfies: $\text{Diff}(M,\omega)$ acts transitively, but not 2-transitively on $M$.
\end{remark}

\begin{proposition}[Nonflat]\label{nflat2} Let $M^{>0}=\{(x^1,x^2,x^3,x^4,x^5,x^6)\in \mathbb R^6~|~x^2>0\}$, $f:\mathbb R^{>0}\to \mathbb R^{>0},~f(x^2)=x^2$ and $\omega^f=  dx^{135} -dx^{146}-dx^{236} + f(x^2)\cdot dx^{245}$. Then $(M^{>0},\omega^f)$ is multisymplectic and of constant linear type. Furthermore it is non-flat and its multisymplectic diffeomorphisms act 1-transitively but not 2-transitively.
\end{proposition}
	
\begin{proof}(Notations as in Example \ref{nflat}) As discussed in Example \ref{nflat}, $(M^{>0},\omega^f)$ has constant linear type and is non-flat. Since the decomposable forms $\omega_1,\omega_2$ fulfilling $\omega=\omega^f=\omega_1+\omega_2$ are unique up to order by Theorem \ref{thm:2pl6}, a multisymplectic diffeomorphism of $(M^{>0},\omega^f)$ preserves or permutes $\omega_1$ and $\omega_2$. Hence it preserves (or reverts the sign of) $\Omega=\omega_1\wedge \omega_2=2\sqrt{x^2}dx^{123456}$ and $d\omega_1$ and thus they preserve the unique bivector field $\xi$ satisfying the equation $\iota_\xi\Omega=d\omega_1$. Moreover, any diffeomorphism preserving $\omega$ also has to preserve or revert the sign of $\iota_\xi\iota_\xi\Omega\in \Omega^2(M^{>0})$. \\
	
	\noindent In our case
	\[
	\xi~~~=~~~\frac{1}{8x^2}\frac{\partial}{\partial x^3}\wedge\frac{\partial}{\partial x^6}~~~+~~~~\frac{1}{8(x^2)^2}\frac{\partial}{\partial x^4}\wedge\frac{\partial}{\partial x^5}
	\] 
	and hence 
	\[
	\iota_\xi\iota_\xi\Omega=\iota_\xi d\omega_1=\frac{1}{16\sqrt{x^2}^5}dx^1\wedge dx^2
	\]
	So, in our case, $	\iota_\xi\iota_\xi\Omega$ is closed, hence its kernel yields a foliation preserved by the multisymplectic diffeomorphisms of $\omega$. This foliation does not depend on the possible sign ambiguities from above. {So, the multisymplectic diffeomorphisms of $\omega$ do not act 2-transitively on $M^{>0}$. However, as it turns out they do act 1-transitively, as we will now see}. For 1-transitivity it suffices to check by a direct comutation, that $\mathfrak X(M^{>0},\omega)$, includes the complete vector fields:
	\[
	\frac{\partial}{\partial x^i},~i\neq 2 \text{ and } x^1\frac{\partial}{\partial x^1}+x^2\frac{\partial}{\partial x^2}- \frac{1}{2}x^4\frac{\partial}{\partial x^4}+ \frac{1}{2}x^5\frac{\partial}{\partial x^5}-x^6\frac{\partial}{\partial x^6}.
	\]
\end{proof}

\noindent In the case of simple Lie groups the proof of Theorem \ref{liegrp} implies the following

\begin{proposition}\label{pdiff1}
		Let $(G,\omega)$ be a compact real simple Lie group with its canonical three-form, as described in Example \ref{lie}. Then $\text{Diff}(M,\omega)$ acts 1-transitively on $G$. However, it acts $2$-transitively if and only if $dim(G)=3$.
\end{proposition}

\begin{proof}
	For the three-dimensional case, the statement follows from Theorem \ref{moser}. For all other dimensions the statement $Aut(\mathfrak g,\omega_e )\subset Aut(\mathfrak g,\langle\cdot,\cdot\rangle)$ from the proof of Theorem \ref{liegrp} implies, that $\text{Diff}(M,\omega)\subset \text{Diff}(M,h)$ (using the left invariance of both tensors). Especially the connected components of the identity satisfy $\text{Diff}_0(M,\omega)\subset \text{Diff}_0(M,h)$. On the other hand $\text{Diff}_0(M,h)=(G\times G)/Z(G)$ (cf. eg. \cite{MR0412354}), acting by $(l_{g_1},r_{g_{2}}^{-1})$, which clearly preserves the biinvariant form $\omega$. Thus 
	\[\text{Diff}_0(M,\omega)= \text{Diff}_0(M,h).\]
	The statement now follows, because $\text{Diff}(M,\omega)/\text{Diff}_0(M,\omega)$ is discrete and $\text{Diff}_0(M,h)$ acts 1-transitively but not 2-transitively.
\end{proof}
\begin{remark}
	Partial results in this direction have been described in \cite{MR3267563}. We also note that $\text{Diff}(M,\omega)\neq \text{Diff}(M,h)$. For the inversion diffeomorphism $\phi:G\to G, \phi(g)=g^{-1}$, we have $T_e \phi(X)=-X$, so $\phi^*h=h$, but $\phi^*\omega=-\omega$.
\end{remark}

\subsection{Non-transitive cases}

A simple necessary criterion for 1-transitivity of a multisymplectic diffeomorphism group is constant linear type. Two areas, where the multisymplectic form is not of the same linear type can not be multisymplectomorphic.  

\begin{example}
	In Example \ref{ex-r6-linear-vary} the spaces $\{x^2<0\}$, $\{x^2=0\}$ and $\{x^2>0\}$ are preserved by (local) multisymplectic diffeomorphisms. Especially, the group $\text{Diff}(M,\omega)$ does not act transitively on $M$.
\end{example}

\noindent Now we will build a compact version of the above example.
\begin{example}[Compact]
	We set $f=sin(2\pi x^2)$. Then the can regard the quotient multisymplectic manifold $M=\mathbb R^6/\mathbb Z^6$ with the form induced by the $\omega$ above, which we call $\tilde \omega$. Again $(M,\tilde \omega)$ does not have constant linear type. Consequently it does not satisfy the Darboux property and the group $\text{Diff}(M,\tilde \omega)$ does not act transitively on $M$.
\end{example}
Using the fact that having the Darboux property is preserved by multisymplectic diffeomorphisms, we can build an example of a multisymplectic manifold of constant linear type, on which the local multisymplectic diffeomorphisms do not act transitively. 
\begin{proposition}[Constant linear type]\label{exclt}
	Let $M^{>0}=\{(x^1,x^2,x^3,x^4,x^5,x^6)\in \mathbb R^6~|~x^2>0\}$, $f:\mathbb R^{>0}\to \mathbb R^{>0}$ be a smooth function satisfying $f|_{]0,1]}=1$ and $f(t)=t$ for $t\geq 2$ and $\omega^f=  dx^{135} -dx^{146}-dx^{236} + f(x^2)\cdot dx^{245}$. Then $(M^{>0},\omega^f)$ is multisymplectic and of constant linear type, but the group of multisymplectic diffeomorphisms does not act transitively on it.
\end{proposition}

\begin{proof}
		The form $\omega^f$ is flat on $\{ x\in M^{>0}~|~x^2\in]0,1[\}$ and non-flat on$\{ x\in M^{>0}~|~x^2>2\}$ by Theorem \ref{thm:2pl6}, Lemma \ref{lem:2pl} and Example \ref{nflat}. 
		As a flat subset can not be equivalent to a non-flat one, this means that $\text{Diff}(M,\omega)$ does not act transitively on $M^{>0}$.
\end{proof}

\section{Observables and Symmetries}
In this section we generalize the notions of an observable and of an (infinitesimal) symmetry from symplectic to multisymplectic geometry, following the ideas of Baez-Rogers for the first and Callies-Fregier-Rogers-Zambon for the second notion (see \cite{MR2892558} and \cite{MR3552542}) The results of this section are either published, or - in the case of the third subsection - can be found on the arXiv preprint server and will be published with full details elsewhere. Thus we do not give proofs here but concentrate on explaining the theory and giving examples. Since the notions treated in this section seem to become central in multisymplectic geometry we felt obliged to, at least, report on them.

\subsection{The Lie $\infty$-algebra of observables}\label{obs1}
One of the key features of a symplectic form $\omega$ on a manifold $M$, is the Lie algebra structure $\{\cdot,\cdot\}_\omega$ it induces on $C^\infty(M)$. The bracket of two functions $f_1,f_2$ is defined by $\{f_1,f_2\}_\omega=\omega(X_{f_1},X_{f_2})$, where $X_{f_i}$ is the unique vector fields satisfying $\iota_{X_{f_i}}\omega=-df$. Trying to generalize the equation defining $X_{f_i}$ to $n$-plectic manifolds with $n>1$, one has to either turn $X_{f_i}$ into multivector fields or to concentrate on differential forms $f_i$ of degree $n{-}1$. Following Baez and Rogers, we choose the latter approach here but observe new subtleties: In general, neither do all $n{-}1$-forms $\alpha$ admit a vector field $X_\alpha$ satisfying the HDW equation $ \iota_{X_\alpha}\omega=-d\alpha$, nor do those admitting such a vector field form a Lie algebra. However, they do form a Lie $\infty$-algebra, cf. \cite{MR2892558, 9783658123901}. 
\begin{definition}
	Let $(M,\omega)$ be an $n$-plectic manifold. We define the  ``Lie $n$-algebra of observables'' $(L_\infty(M,\omega),\{l_k\}_{k\in\{1,...,n+1\}})$ as follows. As a graded vector space it is given by
	\[
	L_\infty(M,\omega)=\bigoplus_{i=0}^{n-2}\Omega^i(M)\oplus \Omega^{n{-}1}_{Ham}(M,\omega),
	\]
	where 
	\[\Omega^{n{-}1}_{Ham}(M,\omega)=\{\alpha~|~d\alpha=-\iota_{X_\alpha}\omega \text{ for some }X_\alpha\in\mathfrak X(M)\}\subset \Omega^{n{-}1}(M).\]
	We turn $L_\infty(M,\omega)$ into a differential graded vector space with differential $l_1=d$ on $\bigoplus_{i=0}^{n-2}\Omega^i(M)$ and differential $l_1=0$ on  $\Omega^{n{-}1}_{Ham}(M,\omega)$. Furthermore, for $1<k\leq n+1$, we introduce maps \[l_k:\Lambda^k\Omega^{n{-}1}_{Ham}(M,\omega)\to L_\infty(M,\omega),\]
	\[l_k(\alpha_1,...,\alpha_{k})=-(-1)^{k(k+1)/2}\iota_{X_{\alpha_k}}...\iota_{X_{\alpha_{1}}}\omega,\]
	where $d\alpha_{i}=-\iota_{X_{\alpha_i}}\omega$. We extend them to operations $\Lambda^kL_\infty(M,\omega)\to L_\infty(M,\omega)$ trivially (i.e. by zero).
\end{definition}
\begin{remark}\label{linfty-properties}
	The operations $\{l_k\}_{k\in\{1,...,n+1\}}$ satisfy the relations
	\[\partial l_k=l_1l_{k+1},\]
for $1<k<n+2$, where $l_{n+2}$ should be interpreted as the zero map. These relations show that $(L_\infty(M,\omega),\{l_k\}_{k\in\{1,...,n+1\}})$ is a Lie $\infty$-algebra, cf. e.g. \cite{9783658123901}. Here $\partial$ denotes the Chevalley-Eilenberg-operator given by
\[
(\partial l_k)(\alpha_1,...,\alpha_{k+1})=\sum_{i<j}(-1)^{i+j}l_k(l_2(\alpha_i,\alpha_j), \alpha_1 ,...,\widehat{\alpha_i},...,\widehat{\alpha_j},...,\alpha_{k+1}),
\] 
where $\widehat \alpha_i$ means that $\alpha_i$ is left out. This operator is defined for any skew-symmetric map with domain a Lie $\infty$-algebra (especially a Lie algebra).
\end{remark}

\begin{example}[Symplectic forms]
	Let $(M,\omega)$ be a 1-plectic (i.e. symplectic) manifold. Then $L_\infty(M,\omega)=C^\infty(M)$, $l_1=0$ and $l_2=\{\cdot,\cdot\}$ is the classical Poisson multiplication of functions on a symplectic manifold.
\end{example}

\begin{example}[Volumes]
	We regard $\mathbb R^n$ with $n\geq 3$ with the standard volume form $\omega=dx^1\wedge ...\wedge dx^n$ as an $n{-}1$-plectic manifold and describe its $l_2$ operation. Let $\alpha,\tilde \alpha$ be $(n{-}2)$-forms. They can be written as follows:
	\[
	\alpha=\sum_{i<j}f_{ij}\iota_{\frac{\partial}{\partial x^j}}\iota_{\frac{\partial}{\partial x^i}}\omega,~~~
	\tilde \alpha=\sum_{i<j}\tilde f_{ij}\iota_{\frac{\partial}{\partial x^j}}\iota_{\frac{\partial}{\partial x^i}}\omega.
	\]
	Hence, we have
	\[d\alpha=-\sum_{i<j}\left(
	\frac{\partial f_{ij}}{\partial x^i}
	\iota_{\frac{\partial}{\partial x^j}}\omega -
	\frac{\partial f_{ij}}{\partial x^j}
\iota_{\frac{\partial}{\partial x^i}}\omega
	\right),~~~~
	X_\alpha=\sum_{i<j}\left(
	\frac{\partial f_{ij}}{\partial x^i}
	{\frac{\partial}{\partial x^j}}-
	\frac{\partial f_{ij}}{\partial x^j}
	{\frac{\partial}{\partial x^i}}
	\right).
	\]

\noindent Setting $f_{ji}=-f_{ij}$ this can be rewritten to 
\[	X_\alpha=\sum_{j}\left(\sum_{k\neq j} 
	\frac{\partial f_{kj}}{\partial x^k}
	\right){\frac{\partial}{\partial x^j}}.
\]

\noindent Consequently, we have
\[
l_2(\alpha,{\tilde \alpha})=\sum_{i<j}\left(
\left(\sum_{k\neq j}  \frac{\partial \tilde f_{kj}}{\partial x^k} \right)
\left(\sum_{l\neq i}  \frac{\partial f_{li}}{\partial x^l} \right)
-
\left(\sum_{l\neq i}  \frac{\partial \tilde f_{li}}{\partial x^l} \right)
\left(\sum_{k\neq j}  \frac{\partial  f_{kj}}{\partial x^k} \right)
\right)\iota_{\frac{\partial}{\partial x^j}}\iota_{\frac{\partial}{\partial x^i}}\omega.
\]
As a consequence of the Darboux theorem for volume forms, (the binary operation of) the observable Lie $(n{-}1)$-algebra of any $(n{-}1)$-plectic $n$-dimensional manifold locally has this form.
\end{example}

\begin{example}[Sums]\label{linfy-sums}
	Let $(M,\omega)$ and $(\tilde M,\tilde \omega)$ be $(n{-}1)$-plectic. There is a (strict) morphism of Lie $\infty$-algebras
	\[L_\infty(M,\omega)\oplus L_\infty(\tilde M,\tilde \omega) \to  L_\infty(M\times \tilde M,\pi_M^*\omega + \pi_{\tilde M}^*\tilde\omega),\]
	given by
	\[
	(\alpha, \tilde \alpha)\mapsto \pi_M^*\alpha + \pi_{\tilde M}^*\tilde\alpha.
	\]
\end{example}

\begin{example}[Products]\label{linftyprod}Given an $n$-plectic manifold $(M,\omega)$ and an $m$-plectic manifold $(\tilde M,\tilde \omega)$ of not necessarily equal degrees, there is a morphism of Lie $\infty$-algebras 
	\[L_\infty(M,\omega)\oplus L_\infty(\tilde M,\tilde \omega) \to  L_\infty(M\times \tilde M,\pi_M^*\omega \wedge \pi_{\tilde M}^*\tilde\omega).\]
constructed in \cite{MR3487554}. It is an extension of the linear map
\[\Omega^{n{-}1}_{Ham}(M\omega)\oplus \Omega^{m-1}_{Ham}(\tilde M,\tilde \omega)\to  \Omega^{n+m}_{Ham}(M\times \tilde M,\pi_M^*\omega \wedge \pi_{\tilde M}^*\tilde\omega),  \]
\[
(\alpha,\tilde \alpha)\mapsto \pi_{ M}^*\alpha\wedge \pi_{\tilde M}^*\tilde \omega + \pi_{ M}^*\omega\wedge \pi_{\tilde M}^*\tilde \alpha.
\]
Unlike the previous case, in general this morphism has ``higher'' components of the type 
	\[\Lambda^k\left(L_\infty(M,\omega)\oplus L_\infty(\tilde M,\tilde \omega)\right) \to  L_\infty(M\times \tilde M,\pi_M^*\omega \wedge \pi_{\tilde M}^*\tilde\omega)\]
also for $k>1$.
\end{example}

\begin{example}[Compact simple Lie groups]
	In the case of connected compact simple Lie groups, we can get a feeling for $L_\infty(G,\omega)$ by regarding the sub-Lie $\infty$-algebra of left-invariant differential forms:
	\[
	L_\infty(G,\omega)^G= \xymatrix{C^\infty(M)^G\ar[r]\ar[d]^\cong& \Omega^1_{Ham}(G,\omega)^G \ar[d]^\cong\\ \mathbb R\ar[r]^0& \mathfrak g^*}. 
	\]
	Identifying $\mathfrak g^*$ with $\mathfrak g$ by use of the Killing form, we can interpret the operations $l_2$ and $l_3$ as follows.
	\[
	l_2:\Lambda^2\mathfrak g \to \mathfrak g,~~l_2(X,Y)=[X,Y],
	\]
	\[
	l_3:\Lambda^3\mathfrak g \to \mathbb R,~~l_3(X,Y,Z)=-\langle X,[Y,Z]\rangle.
	\]
	Thus, $L_\infty(G,\omega)^G\cong (\mathbb R\overset{l_1}\to{\mathfrak g}, l_1=0, l_2=[\cdot, \cdot], l_3= -\langle \cdot ,[\cdot ,\cdot]\rangle)$.
\end{example}

\begin{example}[Abelian $L_\infty$-algebra]In \cite{MR3156115} a 2-plectic 7-manifold with no non-trivial Hamiltonian vector fields is constructed. Thus, $ \Omega^{n{-}1}_{Ham}(M,\omega)=\Omega^{n{-}1}_{cl}(M)$, $l_2=0$ and $l_3=0$. 
	
\end{example}

\subsection{Comoment maps}\label{obs2}
Let $(M,\omega)$ be an $n$-plectic manifold. There is a linear map $L_\infty(M,\omega)\to \mathfrak X(M)$ which maps the binary operation $l_2$ to the Lie bracket of vector fields. It is given by $\alpha\mapsto X_\alpha$ (the unique vector field satisfying the HDW equation $d\alpha=-\iota_{X_\alpha}\omega $) on $\Omega^{n{-}1}_{Ham}(M)$ and zero on all forms of lower degrees. In the language of Lie $\infty$-algebras it is a Lie $\infty$-morphism. Now, given a Lie algebra action $\zeta:\mathfrak g\to \mathfrak X(M)$ (we assume $\zeta$ to be a Lie algebra homomorphism i.e. an infinitesimal right action), we may ask, whether there is a Lie $\infty$-morphism $F:\mathfrak g\to L_\infty(M,\omega)$  lifting this action, i.e. making the following diagram of Lie $\infty$-algebras commute (cf. \cite{MR3552542}).
\[
\xymatrix{
	&L_\infty(M,\omega)\ar[d]\\
	\mathfrak g\ar@{-->}[ur]^{F}\ar[r]^{\zeta}& \mathfrak X(M)
	}
\]
To answer this question, we will first describe the properties a Lie $\infty$-morphism from a Lie algebra to $L_\infty(M,\omega)$ has to satisfy.

\begin{lemma}[\cite{MR3552542,9783658123901}]
	Let $(M,\omega)$ be $n$-plectic and $\mathfrak g$ a Lie algebra. A Lie $\infty$-morphism $F:\mathfrak g\to L_\infty(M,\omega)$ is given by a family of skew-symmetric maps $\{f_i\}_{i=1,...,n}$, where 
	\[
	f_1:\mathfrak g\to \Omega^{n{-}1}_{Ham}(M,\omega),
	\]
	\[
	f_i:\Lambda^i \mathfrak g\to \Omega^{n-i}(M),\text{ for }i>1,
	\]
	satisfying the conditions
	\[
	\partial f_i + l_1 f_{i+1}=-f^*_1l_{i+1}
	\]
	for $i\in\{1,...,n\}$, where  $f_{n+1}$  should be interpreted as the zero map, $\partial$ is the Chevalley-Eilenberg-operator from Remark \ref{linfty-properties} and $f_1^*l_{i+1}$ is given by the pullback formula
	\[
	f_1^*l_{i+1}(\xi_1,...,\xi_{i+1})=l_{i+1}(f_1(\xi_1),...,f_1(\xi_{i+1}).
	\]
\end{lemma}

\begin{definition}
	Let $\zeta:\mathfrak g\to \mathfrak X(M)$ be a Lie algebra homomorphism and $(M,\omega)$ an $n$-plectic manifold. A ``(homotopy) comoment'' for $\zeta$ is a Lie $\infty$-morphism $F=\{f_i\}_{i=1,...,n}$ from $\mathfrak g$ to $L_\infty(M,\omega)$ satisfying $df_1(\xi)=-\iota_{\zeta(\xi)}\omega$, i.e. making the diagram above commute. A comoment for a Lie group action $\phi:M\times G\to M$ is defined as a comoment of its corresponding infinitesimal action. 
\end{definition}

\begin{theorem}[\cite{MR3552542,MR3353735}]\label{comthm}
	Let $(M,\omega)$ be $n$-plectic and $\zeta:\mathfrak g\to \mathfrak X(M)$ be a Lie algebra homomorphism.
	\begin{enumerate}
		\item A comoment for $\zeta$ can not exist unless $L_{\zeta(\xi)}\omega=0$ for all $\xi\in \mathfrak g$, i.e. unless $\zeta$ preserves $\omega$.
		\item Assume $\zeta$ preserves $\omega$. Then for $i\in\{1,...,n+1\}$ the maps $\Lambda^i\mathfrak g\to \Omega^{n+1-i}(M)$ $(\xi_1,...,\xi_i)\mapsto \iota_{\zeta(\xi_i)}...\iota_{\zeta(\xi_1)}\omega$ induce well-defined elements $g_i$ of $H^i(\mathfrak g)\otimes H^{n+1-i}_{dR}(M)$, where $H^*(\mathfrak g)$ is the Lie algebra cohomology of $\mathfrak g$ and $ H^{*}_{dR}(M)$ is the de Rham cohomology of $M$.
		\item A comoment for $\zeta$ exists if and only if $g_i=0$ for $i\in \{1,...,n{+}1\}$.
		\item Especially, if $\omega$ admits a $\zeta$-invariant potential $\eta$, then it has a comoment given by the formulas
		\[
		f_k(\xi_1,...,\xi_k)=(-1)^{k}(-1)^{\frac{k(k+1)}{2}}\iota_{\zeta(\xi_k)}...\iota_{\zeta(\xi_1)} \eta, ~~k\in\{1,...,n\}.
		\]
	\end{enumerate} 
\end{theorem}

\begin{example}[Symplectic]
	Let $(M,\omega)$ be symplectic and $\zeta:\mathfrak g\to \mathfrak X(M)$ a Lie algebra homomorphism. Then the above definition collapses to the classical notion of (equivariant) comoment map. I.e. a multisymplectic comoment is a Lie algebra homomorphism $f=f_1:\mathfrak g\to C^\infty(M)$ satisfying $X_{f(\xi)}=\zeta(\xi)$ for all $\xi\in\mathfrak g$. A necessary condition for the existence such of a comoment is the $\zeta$-invaricance of $\omega$. In such cases the sufficient condition is given by the classes $g_1=(\xi\mapsto \iota_{\zeta(\xi)}\omega)\in H^1(\mathfrak g)\otimes H^1_{dR}(M)$ and $g_2=((\xi_1,\xi_2)\mapsto \iota_{\zeta(\xi_2)}\iota_{\zeta(\xi_1)}\omega)\in H^2(\mathfrak g)\otimes H^0_{dR}(M)$. If $g_1$ vanishes, then a linear (not necessarily equivariant) comoment exists and $g_2$ is the obstruction against equivariance (compare \cite{MR0464312}). 
\end{example}

\begin{example}[Sums and products]
	Let $(M,\omega)$ and $(\tilde M,\omega)$ be multisymplectic manifolds and $\zeta:\mathfrak g\to X(M)$ and $\tilde\zeta:\tilde {\mathfrak g}\to\mathfrak X(M)$ be Lie algebra homomorphisms. Then there is an induced Lie algebra homomorphism $(\zeta, \tilde\zeta):\mathfrak g\oplus \tilde {\mathfrak g}\to \mathfrak X(M\times \tilde M)$. 
\end{example}

\begin{example}[Multicotangent bundles]
Let $G$ be a Lie group and $\vartheta^Q:Q\times G\to  Q$ a right action. For each $g$ the map $\vartheta^Q_g:Q\to Q$ is a diffeomorphism. Then $T\vartheta^Q_g:TQ\to TQ$ is a fiberwise linear diffeomorphism, which makes the following diagram commute:
\begin{align*}
\begin{xy}
\xymatrix{ TQ\ar[d]\ar[r]^{T\vartheta^Q_g}&TQ\ar[d]\\
	Q\ar[r]^{\vartheta^Q_g}&Q
}
\end{xy}
\end{align*}
With the map $T\vartheta^Q_g$ at hand we construct a diffeomorphism $\Lambda^nT^*(\vartheta^Q_g):\Lambda^nT^*Q\to \Lambda^nT^*Q$. Let $\alpha$ be an element of $\Lambda^nT^*Q$ with $\pi(\alpha)=p\in Q$ and $v_1,...,v_n\in T_{\vartheta^Q_gp}Q$.
\begin{align*}
(\Lambda^nT^*(\vartheta^Q_g))(\alpha)(v_1,...,v_n)=\alpha((T_p\vartheta^Q_g)^{-1}v_1,...,(T_p\vartheta^Q_g)^{-1}v_n),
\end{align*}

where $(T_p\vartheta^Q_g)^{-1}:T_{\vartheta^Q_gp}Q\to T_pQ$ is the inverse of the linear map $T_p\vartheta^Q_g:T_pQ\to T_{\vartheta^Q_gp}Q$. Then $\vartheta^M_g:=\Lambda^nT^*(\vartheta^Q_g)$ defines a right action which makes the following diagram commute:

\begin{align*}
\begin{xy}
\xymatrix{ \Lambda^nT^*Q\ar[d]^{\pi}\ar[r]^{\vartheta^M_g}&\Lambda^nT^*Q\ar[d]^{\pi}\\
	Q\ar[r]^{\vartheta^Q_g}&Q
}
\end{xy}
\end{align*}
Thus we have a right action $\vartheta^M$ of $G$ on $M=\Lambda^nT^*Q$ for $1\leq n\leq \mathrm{dim}(Q)$. To see that the action is n-plectic and even strongly Hamiltonian with respect to the canonical n-plectic structure $\omega$ it suffices to show that the n-form $\theta$ is $G$-invariant. Regard $\alpha \in \Lambda^nT^*Q$ and $v_1,...,v_n\in T_\alpha(\Lambda^nT^*Q)$,
\begin{align*}
((\vartheta^M_g)^*\theta)_\alpha(v_1,...,v_n)&=\theta_{\vartheta^M_g\alpha}((T\vartheta^M_g)v_1,...,(T\vartheta^M_g)v_n)\\
&=\vartheta^M_g(\alpha)((T\pi)(T\vartheta^M_g)v_1,...,(T\pi)(T\vartheta^M_g)v_n)\\
&=\vartheta^M_g(\alpha)((T(\pi\circ \vartheta^M_g))v_1,...,(T(\pi\circ\vartheta^M_g))v_n)\\
&=\vartheta^M_g(\alpha)((T(\vartheta^Q_g\circ \pi))v_1,...,(T(\vartheta^Q_g\circ \pi))v_n)\\
&=\alpha((T\vartheta^Q_g)^{-1}(T(\vartheta^Q_g\circ \pi))v_1,...,(T\vartheta^Q_g)^{-1}(T(\vartheta^Q_g\circ \pi))v_n)\\
&=\alpha((T\pi)v_1,...,(T\pi)v_n)\\
&=\theta_\alpha(v_1,...,v_n).
\end{align*}
Thus $(\vartheta^M_g)^*\theta=\theta$ and thus $\omega$ is $G$-invariant (especially $\mathfrak g$-invariant) with an invariant potential. Theorem \ref{comthm} now implies that the action is strongly Hamiltonian with homotopy co-moment map defined via the $G$-invariant potential  $\eta=-\theta$ of $\omega$.
\end{example}

\begin{example}[Subbundles of multicotangent bundles, cf. \cite{MR1244450}]
	Let $V$ be an involutive subbundle of $TQ$ and $\vartheta^Q:Q\times G\to Q$ a right-action preserving $V$, i.e. $T\vartheta_g^Q(V)=V$ for all $g\in G$. Then, for $1\leq i\leq \mathrm{rank}(V)\leq n\leq \mathrm{dim}(Q) $, $\vartheta^M$ preserves the subbundle $\Lambda^n_iT^*Q$, and thus defines a multisymplectic action on it. This action inherits the comoment from $M=\Lambda^nT^*Q$.
\end{example}

\begin{example}[Simple real Lie groups]
Recall from Theorem 21.1. in \cite{MR0024908} that for a semi-simple Lie algebra $\mathfrak g$ we have $H^1(\mathfrak g)=0=H^2(\mathfrak g)=0$ and $0\neq [\omega_e]_{CE}=[\langle[\cdot,\cdot],\cdot\rangle]_{CE}\in H^3(\mathfrak g)$, where $\langle\cdot,\cdot\rangle$ again denotes the Killing form of $\mathfrak g$.

Assume the real connected simple Lie group $G$ acts on itself from the right by $(g,x)\mapsto x\cdot g$. The corresponding infinitesimal action $\zeta$ extends a $\xi \in\mathfrak g$ to a left-invariant vector field $\zeta(\xi)=\xi^l$. This action preserves the multisymplectic structure $\omega$ on $G$. Since $\omega$ is bi-invariant we obtain:
\begin{align*}
\omega(\zeta(\xi_1),\zeta(\xi_2),\zeta(\xi_3))=\omega_e(\xi_1,\xi_2,\xi_3).
\end{align*}
Thus $g_3=[\omega_e]_{CE}\in H^3(\mathfrak g)= H^3(\mathfrak g)\otimes H^0_{dR}(G)$ does not vanish and therefore $\zeta$ can not admit a comoment map.\\

On the other hand, the conjugation right-action $c:G\times G\to G,~c_g(x)=g^{-1}xg$ does admit a comoment map, cf. \cite{MR3552542}.
\end{example}

\subsection{Conserved quantities}\label{obs3}
On a symplectic manifold $(M,\omega)$, given a Hamiltonian function $H\in C^\infty(M)$, a conserved quantity is a function $f\in C^\infty(M)$ satisfying $L_{X_H}f=0$. This construction has been generalised to multisymplectic manifolds in the following manner (\cite{1610.05592}):

\begin{definition}
	Let $(M,\omega)$ be $n$-plectic and $H\in \Omega^{n{-}1}_{Ham}(M,\omega)$. We will call $H$ a ``Hamiltonian form'' in the sequel. A differential form $\alpha\in L_\infty(M,\omega)$ is called a ``conserved quantity'' if $L_{X_H}\alpha$ is exact. It is called ``locally conserved'', if $L_{X_H}\alpha$ is closed and ``strictly conserved'', when $L_{X_H}\alpha=0$. 
\end{definition}
\begin{remark}
	Of course, to define conservedness it suffices if $\alpha$ is any differential form, it does need not to be an element of $ L_\infty(M,\omega)$. However, we will focus on conserved quantities, which are also elements of the $L_\infty$-algebra.
\end{remark}
\begin{example}[Symplectic manifolds] In the case of symplectic manifolds conserved quantities and strictly conserved quantities coincide. $L_{X_H}f$ is exact if and only if it is zero. Locally conserved quantites are functions $f$ such that $L_{X_H}f=-\{H,f\}$ is constant (cf. eg. \cite{MR0289078}). 
\end{example}

\begin{theorem}[\cite{1610.05592}] Let $(M,\omega)$ be $n$-plectic and $H\in \Omega^{n{-}1}_{Ham}(M,\omega)$.
	\begin{enumerate}
		\item Let $\alpha_1,...,\alpha_k$ be locally conserved. Then $l_k(\alpha_1,...,\alpha_k)$ is strictly conserved.
		\item Let $\zeta:\mathfrak g\to \mathfrak{X}(M)$ be an action admitting a comoment $\{f_k\}$. If $L_{\zeta(\xi)}H$ is closed for all $\xi\in\mathfrak g$, then the image of $\partial f_k:\Lambda^{k+1}\mathfrak g\to \Omega^{n-k}(M)$ consists of conserved quantities for $k=1,...,n$, where $\partial f_k$ is defined analogously to $\partial l_k$ in Remark \ref{linfty-properties}.
		\item Let $\Sigma$ be an $d$-dimensional closed manifold, $\sigma_0:\Sigma\to M$ smooth and $\sigma_t=\phi^{X_H}_t\circ \sigma_0$, where $\phi^{X_H}$ is the flow of $X_H$. Then, for a conserved quantity $\alpha\in \Omega^d(M)$, the value of the following integral is independent of $t$: 
		\[\int_{\Sigma}\sigma_t^*\alpha.\]
	\end{enumerate}
\end{theorem}
\begin{remark}
	In \cite{1610.05592} the above statements are refined and analysed in more detail. For instance, 2. can be interpreted as the following statement in terms of Lie algebra homology: Let $\Lambda^\bullet \mathfrak g$ be the exterior algebra over $\mathfrak g$ and $\delta:\Lambda^{k} \mathfrak g\to \Lambda^{k-1} \mathfrak g$ the differential $\delta$ be given on generators by 
	\[
	\delta_k(\xi_1\wedge ...\wedge \xi_{k})=\sum_{i<j} (-1)^{i+j}[\xi_i,\xi_j]\wedge \xi_1\wedge ...\wedge \widehat{\xi_i}\wedge ...\wedge \widehat{\xi_j}\wedge ...\wedge \xi_{k},
	\]
	where the hat symbol indicates omission. The above statement means, that $f_k$ maps $Im(\delta_{k+1})$ to conserved quantities. It does, however, also map $ker(\delta_k)$ to locally conserved quantities. A sufficient condition for $f_k(ker(\delta_k))$ to be conserved is $L_{\zeta(\xi)}H=0$ for all $\xi\in\mathfrak g$.
\end{remark}

\bibliographystyle{halpha}
\bibliography{sample}

\end{document}